\documentclass[12pt]{amsart}
\usepackage[lite]{amsrefs}
\usepackage{amssymb, amsmath, enumitem}
\usepackage{comment}

\newcommand{\R}{\mathbb{R}}

\newcommand{\N}{\mathbb{N}}

\newcommand{\W}{{\bf{W}}}
\newcommand{\I}{{\bf{I}}}
\newcommand{\g}{{\bf{G}}}

\usepackage{hyperref}
\numberwithin{equation}{section}

\theoremstyle{plain}
\newtheorem{Thm}{Theorem}[section]

\newtheorem{Lem}[Thm]{Lemma}

\theoremstyle{definition}
\newtheorem{Def}[Thm]{Definition}
\newtheorem{Rem}[Thm]{Remark}

\theoremstyle{remark}

\begin{document}
%%% ----------------------------------------------------------------------------------------------
%%% ----------------------------------------------------------------------------------------------
\title[Finite Energy Solutions to Nonlinear Equations]{Finite Energy Solutions to Inhomogeneous Nonlinear Elliptic Equations with \\Sub-Natural Growth Terms}
%%% ----------------------------------------------------------------------------------------------
%%% ----------------------------------------------------------------------------------------------
\author{Adisak Seesanea}
\address{Department of Mathematics, University of Missouri, Columbia, MO 65211, USA}
\email{asrt8@mail.missouri.edu}
\thanks{A. S. is partially supported by Development and Promotion of Science and Technology Talents Project, Thailand (DPST)}

\author{Igor E. Verbitsky}
\address{Department of Mathematics, University of Missouri, Columbia, MO 65211, USA}
\email{verbitskyi@missouri.edu}
\thanks{}
%%% ----------------------------------------------------------------------------------------------
%%% ----------------------------------------------------------------------------------------------
\begin{abstract}
We obtain necessary and sufficient conditions for the existence of a positive finite energy solution to the inhomogeneous quasilinear elliptic equation
\[
-\Delta_{p} u = \sigma u^{q} + \mu \quad \text{on} \;\; \mathbb{R}^n
\]
in the sub-natural growth case $0<q<p-1$, where $\Delta_{p}$ ($1<p<\infty$) is the $p$-Laplacian, and $\sigma$, $\mu$ are positive Borel measures on 
$\mathbb{R}^n$. Uniqueness of such a solution is established as well. 

Similar inhomogeneous problems in the sublinear case $0<q<1$ are treated for the fractional Laplace operator  $(-\Delta)^{\alpha}$  in place of $-\Delta_{p}$, on $\mathbb{R}^n$ for $0<\alpha<\frac{n}{2}$, and   on an arbitrary domain $\Omega \subset \R^n$ with positive Green's function in the classical case $\alpha = 1$.
 \end{abstract}
%%% ----------------------------------------------------------------------------------------------
%%% ----------------------------------------------------------------------------------------------
\maketitle
%%% ----------------------------------------------------------------------------------------------
%%% ----------------------------------------------------------------------------------------------
\section{Introduction}\label{sec:intro}
We consider the quasilinear elliptic equation
\begin{equation}\label{main_eq_p-lapacain}
-\Delta_{p} u = \sigma u^{q} +\mu \quad \text{on} \;\; \mathbb{R}^n
\end{equation}
in the sub-natural growth case $0<q<p-1$.

 Here $\Delta_{p}u = \nabla \cdot \left(  \vert \nabla u \vert^{p-2} \nabla u \right)$ is the $p$-Laplacian with $1<p<\infty$, and $\sigma$, $\mu$ are nontrivial nonnegative locally integrable functions on $\R^n$, or more generally, nonnegative locally finite Borel measures on $\R^n$ (in brief $\sigma, \mu \in \mathcal{M}^{+}(\R^n)$) such that $\sigma\not=0$ and
 $\mu\not=0$. The  homogeneous case $\mu=0$ was considered earlier in 
 \cite{CV14a}. However, treating  general  data $\mu\ge 0$ leads to some new phenomena 
 involving possible interaction between $\mu$ and $\sigma$. 
 
 We establish necessary and sufficient conditions on both $\sigma$ and $\mu$ for the existence of a positive finite energy solution $u$ to  \eqref{main_eq_p-lapacain}, so that $\int_{\R^n} |\nabla u |^{p}\;dx < +\infty$ (see Definition \ref{Def_Sol_1}), and prove its uniqueness. 
 
 Our methods are also applicable to the existence problem for positive finite energy solutions $u \in \dot{H}^{\alpha}(\R^n)$, so that $\int_{\R^n} | (-\Delta)^{\frac{\alpha}{2}}u|^{2}\;dx < +
\infty$ (see Definition \ref{Def_Sol_2}), to the fractional Laplace equation 
\begin{equation} \label{main_eq_frac_laplacian}
\left(-\Delta \right)^{\alpha} u = \sigma u^{q} + \mu \qquad \text{in} \;\; \mathbb{R}^n,
\end{equation}
where $0<q<1$ and $\left(-\Delta \right)^{\alpha} $ is the fractional Laplacian with $0<\alpha<\frac{n}{2}$. Uniqueness of such a solution  is proved in the case $0<\alpha\leq1$. 

In the classical case $\alpha=1$, our approach is employed to obtain the existence and uniqueness of a positive finite energy solution $u\in \dot{W}^{1,2}_0(\Omega)$, such that 
$\int_{\Omega} |\nabla u |^{2}\;dx < +\infty$ (see Definition \ref{Def_Sol_1} in the case $p=2$), to the equation
\begin{equation} \label{main_eq_frac_laplacian_domain}
-\Delta u = \sigma u^{q} + \mu \quad \text{in} \;\; \Omega,
\end{equation}
where $0<q<1$ and $\Omega \subset \R^n$ is an arbitrary domain (possibly unbounded) which possesses a positive Green's function. 
The existence of positive weak solutions to \eqref{main_eq_frac_laplacian_domain}, not necessarily of finite energy, is discussed in \cite{QV16}, \cite{QV17}. 

We would like to  point out that the existence and uniqueness of \textit{bounded} solutions to 
\eqref{main_eq_frac_laplacian_domain} 
on $\Omega=\R^n$ in the case where $\mu$ is a nonnegative constant was characterized in 
\cite{BK92}. 

As was mentioned above, this work has been motivated by the results of Cao and Verbitsky \cite{CV14a}, who proved that there exists a unique positive finite energy solution $u$  to the \textit{homogeneous} equation
\begin{equation}\label{homo_eq_p-lapacain}
-\Delta_{p} u = \sigma u^{q}  \quad \text{in} \;\; \mathbb{R}^n,
\end{equation}
where $1<p<\infty$, $0<q<p-1$ and $\sigma \in \mathcal{M}^{+}(\mathbb{R}^n)$, 
if and only if 
\begin{equation}\label{condition_sigma_alpha=1}
\W_{1, p}\sigma \in L^{\frac{(1+q)(p-1)}{p-1-q}}(\R^n, d\sigma).
\end{equation}
 Here, for $1<p<\infty$, $0<\alpha<\frac{n}{p}$ and $\sigma \in \mathcal{M}^{+}(\R^n)$, the (homogeneous) Wolff potential $\W_{\alpha, p}\sigma$  is defined by \cite{HW83} 
\[
\W_{\alpha, p}\sigma(x) = \int_{0}^{\infty} \left[ \frac{\sigma(B(x,r))}{r^{n-\alpha p}}\right]^{\frac{1}{p-1}}\; \frac{dr}{r}, \quad x \in \R^n,
\]
where $B(x,r) = \{ y \in \R^n : |x-y|< r \}$ is a ball centered at $x \in \R^n$ of radius $r >0$. Notice that $\W_{\alpha, p}\sigma=+\infty$ for $\alpha \ge \frac{n}{p}$ unless 
$\sigma=0$. (See \cite{AH96}, \cite{KuMi14} for an overview of Wolff potentials and their applications in Analysis and PDE.)

For $1\le p<\infty$ and a nonempty open set $\Omega \subset \R^n$, by $\dot{W}^{1,p}_{0}(\Omega)$ we denote the homogeneous Sobolev (or Dirichlet) space defined \cite{HKM06}, \cite{MZ97} as the closure of $C_{0}^{\infty}(\Omega)$ with respect to the (semi)norm 
\[
\| u \|_{\dot{W}^{1,p}_{0}(\Omega)} =  \| \nabla u \|_{L^{p}(\Omega)}.
\]
We denote by $W^{-1,p'}(\Omega) = [\dot{W}^{1,p}_{0}(\Omega)]^{*}$ the dual space, where $p' = \frac{p}{p-1}$. If $p<n$ then $W^{-1,p'}(\Omega) \subset \mathcal{D}'(\Omega)$. 

For equation \eqref{main_eq_p-lapacain} on $\R^n$, we will show that condition \eqref{condition_sigma_alpha=1}, combined with the natural assumption that  $\mu$ has finite energy, i.e. (see \cite{AH96}, Sec. 4.5), 
\begin{equation}\label{condition_mu_alpha=1}
\mu \in \dot{W}^{-1,p'}(\R^n) \Longleftrightarrow  \int_{\R^n} \W_{1,p}\mu\;d\mu < +\infty, 
\end{equation}
is necessary and sufficient for the existence of a positive finite energy solution to \eqref{main_eq_p-lapacain}. More precisely, we state our main results as follows.

\begin{Thm}\label{thm_main1}
Let $1<p<n$, $0<q<p-1$, and let $\sigma, \mu \in \mathcal{M}^{+}(\mathbb{R}^n)$.
Then there exists a positive finite energy solution $u \in L^{q}_{loc}(\R^n, d\sigma) \cap \dot{W}^{1,p}_{0}(\R^n)$ to equation 
\eqref{main_eq_p-lapacain} if and only if both \eqref{condition_sigma_alpha=1}  and \eqref{condition_mu_alpha=1} hold.
Moreover, such a solution is unique in $\dot{W}^{1,p}_{0}(\R^n)$. In the case $p \geq n$, there is only a trivial supersolution.
\end{Thm}

In our proof of Theorem \ref{thm_main1}, we show that if \eqref{condition_sigma_alpha=1} holds, then \eqref{condition_mu_alpha=1} implies
a crucial two-weight condition 
\begin{equation}\label{condition_mu_1+q}
{\W}_{1, p}\mu \in L^{1+q}(\R^n, d\sigma),
\end{equation}
which turns out to be necessary for the existence of a positive  solution $u \in L^{q}_{loc}(\R^n, d\sigma) \cap \dot{W}^{1,p}_{0}(\R^n)$ to \eqref{main_eq_p-lapacain}. 

Given \eqref{condition_sigma_alpha=1}, it allows us to deduce the existence of a positive finite energy solution $u$  to equation \eqref{main_eq_p-lapacain} under  assumption \eqref{condition_mu_alpha=1}, by using a positive solution $\tilde u \in L^{1+q}(\R^n, d\sigma)$ to the corresponding nonlinear integral equation
\begin{equation}\label{int_eq_alpha=1}
\tilde u = \W_{1,p}(\tilde u^q d\sigma) + \W_{1,p}\mu \quad d\sigma\text{-}a.e.
\end{equation}
Such a solution $\tilde u$  can be constructed  by an iterative method, provided \eqref{condition_mu_1+q} holds.

As shown in \cite{COV00}, condition \eqref{condition_sigma_alpha=1} is equivalent to the trace inequality 
\begin{equation}\label{trace_inq}
\| \varphi \|_{L^{1+q}(\R^n, d\sigma)} \leq C \| \nabla\varphi \|_{L^{p}(\R^n)}, \quad \forall \varphi \in C_{0}^{\infty}(\R^n),
\end{equation}
where $C$ is a positive constant independent of $\varphi$. 

Moreover, there is an alternative charaterization of \eqref{trace_inq} in terms of capacities due to Maz'ya and Netrusov (see \cite[Sec. 11.6]{Maz11}),
\begin{equation}\label{condition_cap}
\int_{0}^{\sigma(\R^n)} \left[ \frac{r}{\varkappa(\sigma, r)} \right]^{\frac{1+q}{p-1-q}}\;dr < +\infty,
\end{equation}
where $\varkappa(\sigma, r) = \inf \{ \text{cap}_{p}(E)\colon \,  \sigma(E) \geq r, 
E \subset \R^n \, \text{compact}\}$ and $\text{cap}_{p}(\cdot)$ is the $p$-capacity defined,
for  a compact set  $E \subset \R^n$, by 
\[
\text{cap}_{p}(E) =  \inf \big\lbrace \Vert  \nabla u \Vert^{p}_{L^{p}(\R^n)}\colon u \geq 1 \;\; \text{on} \;\; E, \, \, u \in C_{0}^{\infty}(\R^n) \big\rbrace.
\]

Thus, any one of conditions \eqref{condition_sigma_alpha=1},  \eqref{trace_inq},  or \eqref{condition_cap}, combined  with  \eqref{condition_mu_alpha=1}, is necessary and sufficient for the existence of a positive finite energy solution to equation \eqref{main_eq_p-lapacain}. The uniqueness part will be proven by first establishing the minimality of such a solution, and then using convexity of the Dirichlet integrals $\int_{\R^n} |\nabla u |^{p} \;dx $.

Furthermore, we are able to adjust our argument outlined above to obtain analogous results for the fractional Laplace equation \eqref{main_eq_frac_laplacian} as follows.

\begin{Thm}\label{thm_main2}
Let $0<q<1$, $0< \alpha < \frac{n}{2}$, and let $\sigma, \mu \in \mathcal{M}^{+}(\mathbb{R}^n)$. Then there exists a positive finite energy solution 
$u \in L^{q}_{loc}(\R^n, d\sigma) \cap \dot{H}^{\alpha}(\R^n) $ to equation \eqref{main_eq_frac_laplacian} if and only if the following two conditions hold:
\begin{equation}\label{condition_sigma_alpha_2}
\I_{2\alpha}\sigma \in L^{\frac{1+q}{1-q}}(\R^n, d\sigma),
\end{equation}
and
\begin{equation}\label{condition_mu_alpha_2}
\mu \in \dot{H}^{-\alpha}(\R^n).
\end{equation}
Moreover, if $0< \alpha \leq 1$, then such a solution is unique in $\dot{H}^{\alpha}(\R^n)$.
\end{Thm}

Here, for $0<\alpha<\frac{n}{2}$ and $\sigma \in \mathcal{M}^{+}(\R^n)$, we denote by $\I_{2\alpha}\sigma = \W_{\alpha, 2}\sigma$  the Riesz   potential of order $2\alpha$ 
(up to a normalization constant). The  homogeneous Sobolev space $\dot{H}^{\alpha}(\R^n)$ ($0<\alpha<\frac{n}{2}$) can be  defined by means of Riesz potentials, 
\[
\dot{H}^{\alpha}(\R^n) =  \big\lbrace u\colon \,  u = \I_{\alpha}f,\; f \in L^{2}(\R^n) \big\rbrace,
\]
equipped with  norm 
\[
\| u \|_{\dot{H}^{\alpha}(\R^n)}= \| f \|_{L^{2}(\R^n)}.
\]
We denote by $\dot{H}^{-\alpha}(\R^n) = [\dot{H}^{\alpha}(\R^n)]^{*}$ the space 
of distributions dual to  $\dot{H}^{\alpha}(\R^n)$. 

Adapting the previous argument, if  \eqref{condition_sigma_alpha_2}  holds, we first construct a positive solution $\tilde u \in L^{1+q}(\R^n, d\sigma)$ to the integral equation
\begin{equation}\label{int_eq_frac00}
\tilde u = \I_{2\alpha}(\tilde u^{q}d\sigma) + \I_{2\alpha}\mu \quad d\sigma\text{-}a.e.
\end{equation}
using an iterative procedure, under the additional assumption that 
\begin{equation}\label{condition_f_alpha_2}
\I_{2\alpha}\mu \in L^{1+q}(\R^n, d\sigma).
\end{equation}

Using the nontrivial fact \eqref{condition_sigma_alpha_2}\&\eqref{condition_mu_alpha_2}$\Longrightarrow$\eqref{condition_f_alpha_2}, we deduce  the existence of a solution $\tilde u\in L^{1+q}(\R^n, d\sigma)$, 
and then a positive finite energy solution $u$ to equation \eqref{main_eq_frac_laplacian}.

We observe that \eqref{condition_sigma_alpha_2} is equivalent to the trace inequality \cite{COV06}
\begin{equation}\label{trace_inq2}
\big\| \I_{\alpha}g \big\|_{L^{1+q}(\R^n, d\sigma)} \leq C \| g \|_{L^{2}(\R^n)}, \quad \forall g \in L^{2}(\R^n),
\end{equation}
where $C$ is a positive constant independent of $g$. Thus,  condition  \eqref{condition_sigma_alpha_2}, or equivalently \eqref{trace_inq2}, together with condition \eqref{condition_mu_alpha_2} is necessary and sufficient for the existence of a positive finite energy solution to equation \eqref{main_eq_frac_laplacian}. The restriction on the value of $\alpha$ in  the uniqueness result is due to availability \cite{BF14} of a certain convexity property of the Dirichlet integrals 
$
\int_{\R^n}|(-\Delta)^{\frac{\alpha}{2}} u|^{2}\;dx
$
in the case $\alpha\in(0, 1]$. 

We now consider sublinear elliptic equation \eqref{main_eq_frac_laplacian_domain} on arbitrary  domains $\Omega \subset \R^n$ (possibly unbounded) with positive Green's function $G(x,y)$ on $\Omega \times \Omega$. Define the Green potential by 
\[
\g\sigma(x) = \int_{\Omega} G(x,y) \, d\sigma(y), \quad x \in \Omega.
\]
Our main results in this setup are stated in the following theorem.

\begin{Thm}\label{thm_main3}
Let $0<q<1$ and $\sigma, \mu \in \mathcal{M}^{+}(\Omega)$, and let $G$ be Green's function associated with $-\Delta$ on $\Omega$. Then there exists a positive finite energy solution $u \in L^{q}_{loc}(\Omega, d\sigma) \cap \dot{W}^{1,2}_{0}(\Omega)$ to equation \eqref{main_eq_frac_laplacian_domain} if and only if the following two conditions hold: 
\begin{equation}\label{condition1_domain} 
\g\sigma \in L^{\frac{1+q}{1-q}}(\Omega, d\sigma), 
 \end{equation}
 and 
\begin{equation}\label{condition3_domain}
\mu \in \dot{W}^{-1,2}(\Omega).
\end{equation} 
Moreover, such a solution is unique in $\dot{W}^{1,2}_{0}(\Omega)$.
\end{Thm}

Recently, it has been shown in \cite{V17a} that \eqref{condition1_domain} is equivalent to the weighted norm inequality for Green's potentials, 
\begin{equation}\label{weighted_norm_ineq_special0}
\big\| \g(f d\sigma) \big\|_{L^{1+q}(\Omega, d\sigma)} \leq C \| f \|_{L^{\frac{1+q}{q}}(\Omega, d\sigma)}, \quad \forall f \in L^{\frac{1+q}{q}}(\Omega, d\sigma)
\end{equation}
where $C$ is positive constant independent of $f$. Therefore, condition \eqref{condition1_domain}, or equivalently  \eqref{weighted_norm_ineq_special0}, together with condition \eqref{condition3_domain} turns out to be  necessary and sufficient for the existence of a positive finite energy solution to equation \eqref{main_eq_frac_laplacian_domain}.

Our argument is based on the results in \cite{V17a} mentioned above, along with a new element that given \eqref{condition1_domain}, condition \eqref{condition3_domain} yields 
\begin{equation}\label{condition2_domain}
\g \mu \in L^{1+q}(\Omega, d\sigma).
 \end{equation}
As before, when \eqref{condition1_domain} holds, this allows us to construct a positive finite energy solution to equation \eqref{main_eq_frac_laplacian_domain} by using an auxiliary solution $\tilde u \in L^{1+q}(\Omega, d\sigma)$ to the corresponding integral equation
\begin{equation}\label{int_eq_domain00}
\tilde u = \g (\tilde u^q d\sigma) + \g \mu \quad d\sigma\text{-}a.e.
\end{equation}

Analogues of Theorems \ref{thm_main2} and \ref{thm_main3} for the fractional Laplacian $\left(-\Delta \right)^{\alpha}$ on domains $\Omega$ with Green's function $G$ in the case $0<\alpha<1$ (see [3]) will be considered elsewhere. There are also some analogous results (less precise at the boundary $\partial \Omega$)  for equation \eqref{main_eq_p-lapacain} involving the $p$-Laplace operator
 in domains $\Omega\subset \R^n$;  see Remark \ref{phuc-v} below.

This paper is organized as follows. In Sec. 2, we recall the necessary mathematical background, together with preliminary results concerning
quasilinear equations and nonlinear potentials. In Sections 3, 4 and 5, we establish explicit necessary and sufficient conditions for the existence of positive finite energy solutions to equations \eqref{main_eq_p-lapacain}, \eqref{main_eq_frac_laplacian}, and \eqref{main_eq_frac_laplacian_domain}, respectively. Uniqueness results for such solutions are discussed in Sec. 6.

Throughout, the letters $c$ and $C$ denote various positive constants whose value may change from one place to another. 
%%% ----------------------------------------------------------------------------------------------
%%% ----------------------------------------------------------------------------------------------
\section{Preliminaries}
Let $\Omega \subseteq \R^n$ be a domain (nonempty open connected set). We denote
by $\mathcal{M}^{+}(\Omega)$ the set of all nontrivial nonnegative locally finite Borel measures in $\Omega$, and by $C^{\infty}_{0}(\Omega)$ the set of all smooth compactly supported functions in $\Omega$. 

For $1 \leq p<\infty$ and $\sigma\in  \mathcal{M}^{+}(\Omega)$,
we denote by $L^{p}(\Omega, d\sigma)$ the space of all real-valued measurable 
 functions $u$ on $\Omega$ such that 
\[
\Vert u \Vert_{L^{p}(\Omega, d\sigma)} = \left( \int_{\Omega}|u|^{p}\;d\sigma \right)^{\frac{1}{p}}<\infty.
\]
The corresponding local space $L_{loc}^{p}(\Omega, d\sigma)$ consists of real-valued measurable functions $u$ on $\Omega$ such that the restriction $u_{|K} \in L^{p}(K, d\sigma)$ for every compact subset $K \subset \Omega$. When $\sigma$ is ($n$-dimensional) Lebesgue measure, $d\sigma = dx$, we write $L^{p}(\Omega)$ and $L_{loc}^{p}(\Omega)$, respectively. 

For $1 \leq p < \infty$, the Sobolev space $W^{1, p}(\Omega)$ consists of all functions $u \in L^p(\Omega)$ such that $|\nabla u | \in L^p(\Omega)$, where $\nabla u $ is the vector of distributional (or weak) partial derivatives of $u$ of order 1. The norm on $W^{1, p}(\Omega)$ is given by 
\[
\| u \|_{W^{1, p}(\Omega)} = \| u \|_{L^{p}(\Omega)} + \| \nabla u \|_{L^{p}(\Omega)}.
\]
The corresponding local space denoted by  $W_{loc}^{1, p}(\Omega)$ is  the space of all functions $u$ in $\Omega$ such that the restriction $u_{| F} \in W^{1, p}(F)$ for every relatively compact open subset $F \subset \Omega$. 

The Sobolev space $W_{0}^{1, p}(\Omega)$ is defined as the closure of $C_{0}^{\infty}(\Omega)$ in $W^{1, p}(\Omega)$. It is easy to see that $W_{0}^{1, p}(\R^n) = W^{1, p}(\R^n)$. The homogeneous version of  $W_{0}^{1, p}(\Omega)$, called the homogeneous Sobolev space (or Dirichlet space), denoted by $\dot{W}^{1,p}_{0}(\Omega)$, is defined as the closure of $C_{0}^{\infty}(\Omega)$ with respect to the seminorm 
\[
\| u \|_{\dot{W}^{1,p}_{0}(\Omega)} = \| \nabla u \|_{L^{p}(\Omega)}.
\]
That is, $\dot{W}^{1,p}_{0}(\Omega)$ is the set of all functions $u \in W^{1,p}_{loc}(\Omega)$ such that $|\nabla u| \in L^{p}(\Omega)$
for which there exists a sequence $\{ \varphi_{j}\}_{1}^{\infty} \subset C_{0}^{\infty}(\Omega)$ such that $\| \nabla u - \nabla \varphi_{j} \|_{L^{p}(\Omega)} \rightarrow 0 $ as $j \rightarrow \infty$. When $1 < p < n$, the dual space to $\dot{W}_{0}^{1,p}(\Omega)$ denoted by $W^{-1,p'}(\Omega)$, is the space of distributions 
$\mu\in \mathcal{D}'(\Omega)$ such that
\[
\| \mu \|_{W^{-1,p'}(\Omega)} = \sup \frac{|\langle \mu, u \rangle|}{ \| u \|_{\dot{W}_{0}^{1,p}(\Omega)}} < +\infty,
\]
where the supremum is taken over all nontrivial functions $u \in C^\infty_{0}(\Omega)$. Here, $p' = \frac{p}{p-1}$ is the H\"{o}lder conjugate of $p$. For a measure $\mu \in \mathcal{M}^{+}(\Omega)$, $\mu \in \dot{W}^{-1,p'}(\Omega)$ if and only if there exists a positive constant $C$ such that 
\[
\Big| \int_{\Omega} \varphi \; d\mu \Big| \leq C \left( \int_{\Omega} |\nabla \varphi |^{p} \;dx \right)^{\frac{1}{p}}, 
\quad  \forall \varphi \in C_{0}^{\infty}(\Omega).
\]

For $0 < \alpha < n $, the Riesz potential $\I_{\alpha}f$ of a function $f \in L^{1}_{loc}(\R^n)$ is defined by 
\[
\I_{\alpha}f(x) = (-\Delta)^{-\frac{\alpha}{2}}f(x)= \gamma(\alpha, n)\int_{\R^n} \frac{f(y)}{|x-y|^{n-\alpha}}\;dy, \quad x \in \R^n,
\]
where $\gamma(\alpha, n) = \frac{\Gamma \left( \frac{n-\alpha}{2}\right)}{\pi^{\frac{n}{2}}2^{\alpha}\Gamma \left( \frac{\alpha}{2}\right)}$ is a normalization constant. 

Observe that,  for $f \in L^{p}(\R^n)$, $1< p < \frac{n}{\alpha}$, the Riesz potential $\I_{\alpha}f$ is well-defined and finite $(\alpha, p)$-quasi everywhere (briefly, q.e.), meaning everywhere except for a set of $(\alpha, p)$-capacity zero (see \cite{AH96}). Moreover, $\I_{\alpha}f$ is $(\alpha, p)$-quasicontinuous (in brief, quasicontinuous) which means that, for every $\epsilon > 0$, there is an open set $G \subset \R^n$ such that $\text{cap}_{\alpha, p}(G) < \epsilon$ and the restriction $\I_{\alpha}f |_{G^c}$ is continuous on $G^{c}$. Here the $(\alpha, p)$-capacity of $E\subset \R^n$  is defined by
\[
\text{cap}_{\alpha, p}(E)\!\colon=  \, \inf \big\lbrace \Vert u \Vert^{p}_{L^{p}(\R^n)} : \I_{\alpha} u \geq 1\;\text{on}\; E, \, u \geq 0\; \text{a.e.}, u \in L^{p}(\R^n) \big\rbrace.
\]
Note that Lebesgue measure is
absolutely continuous with respect to the $(\alpha, p)$-capacity, i.e., each set of $(\alpha, p)$-capacity zero has Lebesgue measure zero.

In a similar manner, the Riesz potential $\I_{\alpha}\sigma$ of order $\alpha\in (0, n)$ of a measure $\sigma \in \mathcal{M}^{+}(\R^n)$ is defined  by 
\[
\I_{\alpha}\sigma(x) =(-\Delta)^{-\frac{\alpha}{2}}\sigma(x) = (n-\alpha) \gamma(\alpha, n) \int_{0}^{\infty} \frac{\sigma \left( B(x,r) \right)}{r^{n-\alpha}} \;\frac{dr}{r}, \quad x \in \R^n.
\]
Henceforth, the normalization constant will be dropped for the sake of convenience.

For $1<p<\infty $ and $0 < \alpha < \frac{n}{p}$,  the fractional homogeneous Sobolev space is defined by (see \cite{St70})
\[
\dot{L}^{\alpha, p}(\R^n) =  \big\lbrace u\colon \,  u = \I_{\alpha}f,\; f \in L^{p}(\R^n) \big\rbrace,
\]
equipped with  norm 
\[
\| u \|_{\dot{L}^{\alpha, p}(\R^n)}= \| f \|_{L^{p}(\R^n)}.
\] 
Clearly, $\dot{W}_{0}^{1, p}(\R^n)= \dot{L}^{1, p}(\R^n)$. 
In the case $p=2$, we use the notation $\dot{L}^{\alpha,2}(\R^n) = \dot{H}^{\alpha}(\R^n)$. It is well-known that when $0< \alpha < 1 $, $\| u \|_{\dot{H}^{\alpha}(\R^n)}$ is equivalent to the Gagliardo seminorm 
\[
\left( \int_{\R^n}\int_{\R^n} \frac{\vert u (x) - u (y)\vert^{2} }{\vert x-y \vert^{n+2\alpha}}\;dxdy \right)^{\frac{1}{2}},
\]
see, for example, \cite[Sec. 3.5]{AH96}. 

The dual space of $\dot{H}^{\alpha}(\R^n)$ 
for $0<\alpha<\frac{n}{2}$, 
denoted by $\dot{H}^{-\alpha}(\R^n)$, consists of distributions $\mu \in \mathcal{D}'(\R^n)$ such that 
\[
\| \mu \|_{\dot{H}^{-\alpha}(\R^n)} = \sup \frac{| \langle \mu, u \rangle |}{\| u \|_{\dot{H}^{\alpha}(\R^n)}} < + \infty,
\]
where the supremum is taken over all nontrivial functions $u \in C^{\infty}_0(\R^n)$. Thus, by duality, for a measure $\mu \in \mathcal{M}^{+}(\R^n)$, we have $\mu \in \dot{H}^{-\alpha}(\R^n)$ if and only if  $ \big\| \I_{\alpha}\mu \big\|_{L^{2}(\R^n)} < +\infty $, or equivalently 
$\int_{\R^n}  \I_{2\alpha}\mu \, d \mu<\infty$.

Let $1<p<\infty$, $0< \alpha < \frac{n}{p}$ and $\sigma \in \mathcal{M}^{+}(\R^n)$. The (homogeneous) Wolff potential $\W_{\alpha, p}\sigma$  is defined by (see \cite{AH96}, \cite{KuMi14})
\[
\W_{\alpha, p}\sigma(x) = \int_{0}^{\infty} \left[ \frac{\sigma(B(x,r))}{r^{n-\alpha p}}\right]^{\frac{1}{p-1}}\; \frac{dr}{r}, \quad x \in \R^n,
\]
where $B(x,r) = \{ y \in \R^n : |x-y|< r \}$ is a ball centered at $x \in \R^n$ of radius $r >0$.

Observe that $\W_{\alpha, p}\sigma$ is always positive since $\sigma \not\equiv 0$. Moreover, either $\W_{\alpha, p}\sigma \equiv +\infty$ or $\W_{\alpha, p}\sigma < +\infty$ a.e. In other words,  $\W_{\alpha, p}\sigma < +\infty$ a.e.  if and only if  $\W_{\alpha, p}\sigma(x_0) < +\infty$ for some $x_{0} \in \R^n$. 

In the linear case, when $p=2$, $\W_{\alpha, 2}\sigma = \I_{2\alpha}\sigma$, and in particular, $\W_{1, 2}\sigma = \I_{2}\sigma$
is the Newtonian potential. 

The energy of $\sigma$ is given by
\[
\mathcal{E}_{\alpha, p}(\sigma) = \big\Vert \I_{\alpha}\sigma \big\Vert^{p'}_{L^{p'}(\R^n)}.
\]
The fundamental Wolff's inequality, see \cite[Sec. 4.5]{AH96}, provides a certain estimate of the energy by means of the corresponding Wolff potential:
\begin{equation}\label{Wolff's ineq}
C^{-1} \mathcal{E}_{\alpha, p}(\sigma) \leq \int_{\R^n} \W_{\alpha, p} \sigma \;d\sigma \leq C \mathcal{E}_{\alpha, p}(\sigma),
\end{equation}
where $C=C(\alpha, n, p) \geq 1$. Consequently, 
\[
\W_{\alpha, p} \sigma \in L^{1}(\R^n, d\sigma)
\;\; \Longleftrightarrow \;\; 
 \mathcal{E}_{\alpha, p}(\sigma) < +\infty. 
\]
More generally, it was shown in \cite{COV00} (see also \cite{COV06}) that for $0 \leq q < p$, $p>1$, 
\[
\W_{\alpha, p} \sigma \in L^{\frac{q(p-1)}{p-q}}(\R^n, d\sigma)
\]
is equivalent to the trace inequality 
\begin{equation}\label{trace_ineq}
\left( \int_{\R^n} \big| \I_{\alpha}g \big|^{q}\;d\sigma \right)^{\frac{1}{q}} \leq C \left( \int_{\R^n} \big|  g \big|^{p}\;dx \right)^{\frac{1}{p}}, \quad \forall g \in L^{p}(\R^n),
\end{equation}
where $C$ is a constant independent of $g$. When $\alpha = k<\frac{n}{2}$ is a positive integer, \eqref{trace_ineq} is equivalent to the
generalized Sobolev inequality
\begin{equation}\label{trace_ineq2}
\left( \int_{\R^n} \big| g \big|^{q}\;d\sigma \right)^{\frac{1}{q}} \leq C \left( \int_{\R^n} \big| \nabla^{k} g \big|^{p}\;dx \right)^{\frac{1}{p}}, \quad \forall g \in C_{0}^{\infty}(\R^n),
\end{equation}
where $C$ is a constant independent of $g$.

For $1<p<\infty$, the $p$-Laplacian $\Delta_{p}$ is defined by
\[
\Delta_{p}u = \nabla \cdot \left(  \vert \nabla u \vert^{p-2} \nabla u \right), \quad  u \in W^{1,p}_{loc}(\Omega),
\] 
in the distributional sense, i.e., for every $\varphi \in  C_{0}^{\infty}(\Omega)$,
\[
\langle \Delta_{p}u, \varphi \rangle = \langle \nabla \cdot \left(  \vert \nabla u \vert^{p-2} \nabla u \right), \varphi \rangle = -\int_{\Omega} \vert \nabla u \vert^{p-2} \nabla u \cdot \nabla \varphi \;dx.
\]

\begin{Def}\label{Def_Sol_1}
Let $1<p<\infty$, $0<q<p-1$ and $\sigma, \mu \in \mathcal{M}^{+}(\Omega)$.
A function $u$ is said to be a finite energy solution to the equation
\begin{equation}\label{eq_domain_p}
-\Delta_{p} u = \sigma u^{q} +\mu \quad \text{in} \;\; \Omega
\end{equation}
 if $u \in L_{loc}^{q}(\Omega, d\sigma) \cap \dot{W}_{0}^{1,p}(\Omega)$, $u \geq 0$ $d\sigma\text{-}a.e.$ and
\[
\int_{\Omega} \vert \nabla u \vert^{p-2} \nabla u \cdot \nabla \varphi \;dx = \int_{\Omega} \varphi u^{q}\; d\sigma + 
\int_{\Omega} \varphi \;d\mu, \quad \varphi \in  C_{0}^{\infty}(\Omega).
\]
\end{Def}

We shall extend the notion of distributional solutions $u$ to equation
\eqref{eq_domain_p}, for $u$ not necessarily belonging to $W_{loc}^{1,p}(\Omega)$. We will understand such solutions in the potential-theoretic sense using $p$-superharmonic functions, which is equivalent to the notion of locally renormalized solutions in terms of test functions, see \cite{KKT09}.

A function $u \in W_{loc}^{1,p}(\Omega)$ is said to be $p$-harmonic if 
$u$ satisfies the $p$-Laplace equation
\[
-\Delta_{p}u = 0 \qquad \text{in} \;\; \Omega
\]
in the distributional sense. Note that every $p$-harmonic function has a continuous representative which coincides with $u$ a.e., see \cite{HKM06}. A function $u: \Omega \rightarrow (-\infty, +\infty]$ is $p$-superharmonic if 
$u$ is lower semicontinuous in $\Omega$, $u \not\equiv +\infty $ in each component of $\Omega$, and whenever $D$ is an open relatively compact subset of  $\Omega$ and $h \in C(\overline{D})$ is $p$-harmonic in $D$ with $h \leq u$ on $\partial D$, then $h \leq u$ on $D$. Also note that every $p$-superharmonic function $u$ in $\Omega$ has a quasicontinuous representative which coincides with $u$ $p$-quasi-everywhere in $\Omega$ (briefly, q.e.), i.e., everywhere except for a set of $p$-capacity zero. Here, the $p$-capacity of a compact set $E \subset \Omega$ is defined by 
\[
\text{cap}_{p}(E) :=  \inf \big\lbrace \Vert  \nabla u \Vert^{p}_{L^{p}(\Omega)}\colon \, \,  u \geq 1 \;\; \text{on} \;\; E, u \in C_{0}^{\infty}(\Omega) \big\rbrace.
\]
Notice that $\text{cap}_{p}(E)$ is equivalent to $\text{cap}_{1, p}(E)$ for compact sets $E \subset \Omega$.

A $p$-superharmonic function $u \geq 0$ does not necessarily belong to $W^{1,p}_{loc}(\Omega)$, but its truncation 
\[
T_{k}(u):= \text{min}(u,k)
\]
does for every $k \in \N$. Moreover, each $T_k(u)$ is a supersolution, i.e.,
\[
-\nabla \cdot \left(  \vert \nabla T_k(u)\vert^{p-2} \nabla T_k(u) \right) \geq 0
\]
in the distributional sense. The generalized (or weak) gradient of a $p$-superharmonic function $u$ is defined by (see \cite{HKM06}):  
\[
Du = \lim_{k \rightarrow \infty} \nabla \left( T_k(u) \right).
\]

Let $u$ be a $p$-superharmonic function in $\Omega$. Then $\vert Du \vert^{p-1}$ and consequently $\vert Du \vert^{p-2} Du$, are of class  $L^{r}_{loc}(\Omega)$ for every $1 \leq r < \frac{n}{n-1}$, see \cite{KM92}. This allows us to define a nonnegative distribution $-\Delta_{p}u $ by
\[
-\langle \Delta_{p}u, \varphi \rangle = \int_{\Omega} \vert Du \vert^{p-2} Du \cdot \nabla \varphi \;dx, \quad \varphi \in  C_{0}^{\infty}(\Omega).
\]
Thus, by the Riesz Representation Theorem, there exists a unique measure $\omega[u] \in \mathcal{M}^{+}(\Omega)$ so that $ -\Delta_{p}u = \omega[u]$. The measure $\omega[u]$ is called the Riesz measure of $u$.

\begin{Def}
For $\omega \in \mathcal{M}^{+}(\Omega)$, a function $u$ is said to be a solution to the equation 
\[
- \Delta_{p}u = \omega \qquad \text{in} \;\; \Omega
\]
(in the potential-theoretic sense) if $u$ is $p$-superharmonic in $\Omega$ and $\omega[u] = \omega$.

Thus, for $\sigma, \mu \in \mathcal{M}^{+}(\Omega)$, a function $u$ is said to be a solution to equation \eqref{eq_domain_p}
(in the potential-theoretic sense) if $u$ is $p$-superharmonic in $\Omega$ so that $u \in L^{q}_{loc}(\Omega, d\sigma)$ and $d\omega[u] = u^{q} d\sigma + d\mu$.

A supersolution to \eqref{eq_domain_p} is a nonnegative $p$-superharmonic function $u$ in $\Omega$ so that $u \in L^{q}_{loc}(\Omega, d\sigma)$ and 
\[
\int_{\Omega} \vert Du \vert^{p-2} Du \cdot \nabla \varphi \;dx 
\geq \int_{\Omega} u^q \varphi \;d\sigma + \int_{\Omega} \varphi \;d\mu, \quad \varphi \in C_0^{\infty}(\Omega) \;\; \text{with} \;\; \varphi \geq 0.
\] 
\end{Def}

Note that if $u \in W_{loc}^{1,p}(\Omega)$ is a solution (or supersolution) to equation \eqref{eq_domain_p}, then the generalized gradient $Du$ coincides with the regular gradient $u$. Thus $u$ is the usual distributional solution (or supersolution, repectively).

The following weak continuity result, see \cite{TW02}, will be used to prove the existence of $p$-superharmonic solutions to quasilinear equations.

\begin{Thm} [\cite{TW02}] \label{weak_cont_p-Laplacian}
Suppose $\lbrace u_{j} \rbrace_{1}^{\infty}$ is a sequence of nonnegative $p$-superharmonic functions in $\Omega$ such that $u_j \rightarrow u$ a.e. as $j \rightarrow \infty$, where $u$ is a $p$-superharmonic function in $\Omega$. Then $\omega[u_{j}]$ converges weakly to $\omega[u]$, that is,
\[
\lim_{j \rightarrow \infty} \int_{\Omega} \varphi \; d\mu[u_j] = \int_{\Omega} \varphi \; d\mu[u]
\]
for all $\varphi \in C_{0}^{\infty}(\Omega)$.
\end{Thm}

We shall use the following lower bounds for supersolutions.

\begin{Thm}[\cite{CV14a}] \label{lower_ptwise_est_homo} 
Let $1<p<n$, $0<q<p-1$ and $\sigma \in \mathcal{M}^{+}(\mathbb{R}^n)$.
Suppose $u$ is a nontrivial supersolution to equation \eqref{homo_eq_p-lapacain}. Then $u$ satisfies the inequality 
\[
u \geq c \left( \W_{1, p}\sigma \right)^{\frac{p-1}{p-1-q}}
\]
where $c = c(n, p, q) > 0$.
\end{Thm}

\begin{Thm} [\cite{CV14b}] \label{lower_ptwise_est}  
Let $1<p<n$, $0<q<p-1$, $0<\alpha<\frac{n}{p}$ and $\sigma \in \mathcal{M}^{+}(\mathbb{R}^n)$. Suppose $u \in L^{q}_{loc}(\mathbb{R}^n, d\sigma)$ satisfying
\[
u \geq  \W_{\alpha, p} (u^{q}d\sigma) \quad d\sigma\text{-}a.e.
\]
Then, $u$ satisfies the inequality
\[
u \geq c \left( \W_{\alpha, p}\sigma \right)^{\frac{p-1}{p-1-q}} \quad d\sigma\text{-}a.e.,
\]
where $c = c(\alpha, n, p, q) > 0$.
\end{Thm}

The following important result, \cite{KM94}, is concerned with pointwise estimate of nonnegative $p$-superharmonic functions in terms of Wolff's potential.

\begin{Thm} [\cite{KM94}]\label{pointwise_est_p-superharmonic} 
Let $1<p<n$ and $\omega \in \mathcal{M}^{+}(\mathbb{R}^n)$ Suppose 
$u$ is a $p$-superharmonic function in $\mathbb{R}^n$ 
satisfying
\[
\begin{cases}
-\Delta_{p} u = \omega \quad \text{in} \;\; \mathbb{R}^n, \\ 
\liminf\limits_{\vert x \vert \rightarrow \infty} u(x) = 0
\end{cases}
\]
Then 
\[
K^{-1}\W_{1,p} \omega \leq u \leq K \W_{1,p}\omega,
\]
where $K=K(n,p) \geq 1$.
\end{Thm}

The next three lemmas are discussed in \cite{CV14a}, which will be used in our arguments occasionally.

\begin{Lem} [\cite{CV14a}] \label{converse_existence_thm} 
Let $1<p<n$, $0<q<p-1$ and $\sigma \in \mathcal{M}^{+}(\mathbb{R}^n)$.
Suppose there exists a nontrivial supersolution $u \in L_{loc}^{q}(\mathbb{R}^n, d\sigma) \cap \dot{W}_{0}^{1,p}(\mathbb{R}^n)$ to equation \eqref{homo_eq_p-lapacain} Then 
\[
-\Delta_{p} u \in W^{-1,p'}(\mathbb{R}^n) \cap \mathcal{M}^{+}(\mathbb{R}^n) \quad \text{and} \quad u \in L^{1+q}(\mathbb{R}^n, d\sigma),
\]
for a quasicontinuous representative of $u$. Consequently,  \eqref{condition_sigma_alpha=1} holds.
\end{Lem}

\begin{Lem} [\cite{CV14a}] \label{regulartiy_of_measure} 
Suppose $u \in L^{1+q}(\mathbb{R}^n, d\sigma)$ is a nontrivial supersolution to the integral equation 
\begin{equation}
u = \W_{1,p}(u^{q}d\sigma) \quad d\sigma\text{-}a.e.
\end{equation}
Then
\[
u^{q}d\sigma \in W^{-1,p'}(\mathbb{R}^n) \cap \mathcal{M}^{+}(\mathbb{R}^n).
\]
\end{Lem}

\begin{Lem} [\cite{CV14a}] \label{weak_comparison_principle} 
Let $\mu, \omega \in W^{-1,p}(\mathbb{R}^n) \cap \mathcal{M}^{+}(\mathbb{R}^n)$.
Suppose $u, v \in \dot{W}^{1,p}_{0}(\mathbb{R}^n)$ are solutions to the equations 
\[
-\Delta_{p}u = \mu \quad \text{in} \;\; \mathbb{R}^n \qquad \text{and} \qquad -\Delta_{p}v = \omega \quad \text{in} \;\; \mathbb{R}^n,
\]
respectively. If $\mu \leq \omega$, then $u \leq v$ q.e.
\end{Lem}

The following theorem is due to Brezis and Browder \cite{BrB79} (cf. \cite[Theorem 2.39]{MZ97}).

\begin{Thm} \label{representative_thm} 
Let $1<p<n$ and $\mu \in W^{-1,p'}(\Omega) \cap \mathcal{M}^{+}(\Omega)$. Then for any $u \in \dot{W}_{0}^{1,p}(\Omega)$ we have  $u \in L^{1}(\Omega, d\mu)$ and 
\[
\langle \mu, u\rangle = \int_{\Omega} u \;d\mu,
\]
for a quasicontinuous representative of $u$.
\end{Thm}

We shall use the following facts, which are discussed in \cite[Secs. 2.1 - 2.2]{MZ97}.

\begin{Rem}\label{Rem0}
Let $1<p<n$ and $\omega \in W^{-1,p'}(\mathbb{R}^n) \cap \mathcal{M}^{+}(\mathbb{R}^n)$. There exists a unique $p$-superharmonic solution $u \in \dot{W}^{1,p}_{0}(\R^n)$ to the equation
\[
-\Delta_{p}u = \omega  \quad \text{in} \;\; \mathbb{R}^n
\]
in the distributional sense. Moreover, $u \in L^{1}(\mathbb{R}^n, d\omega)$ and 
\[
\langle \omega, u\rangle = \int_{\mathbb{R}^n} u \;d\omega = \Vert u \Vert_{\dot{W}_{0}^{1,p}(\mathbb{R}^n)}^{p} = \Vert \omega \Vert_{W^{-1,p'} (\mathbb{R}^n)}^{p'} 
=\Vert I_{1}\omega \Vert_{L^{p'}(\mathbb{R}^n)}^{p'} = \mathcal{E}_{1,p}(\omega)
\]
for a quasicontinuous representative of $u$.
\end{Rem}

We will need the next lemma which shows that if there exists a nontrivial supersolution $u \in L^{q}_{loc}(\R^n, d\sigma)$ to the integral equation
\begin{equation}\label{int_eq_alpha_p_homo}
u = \W_{\alpha,p}(u^{q}d\sigma) \quad d\sigma\text{-}a.e.,
\end{equation}
then $\sigma$ must be absolutely continuous with respect to $\text{cap}_{\alpha, p}(\cdot)$.

\begin{Lem} [\cite{CV14b}]\label{lemma_cap}
Let $1<p<\infty$, $0<\alpha<\frac{n}{p}$ and $0<q<p-1$ and $\sigma \in \mathcal{M}^{+}(\mathbb{R}^n)$. Suppose there exists a nontrivial supersolution 
$u \in L^{q}_{loc}(\R^n, d\sigma)$ to \eqref{int_eq_alpha_p_homo}. Then there exists a positive constant $c$ such that
\[
\sigma(E) \leq c \left[ \textnormal{cap}_{\alpha, p}(E) \right]^{\frac{q}{p-1}} 
\left( \int_{E} u^q \;d\sigma  \right)^{\frac{p-1-q}{p-1}}
\]
for all compact sets $E \subset \R^n$. 
\end{Lem}

Consequently, if \eqref{homo_eq_p-lapacain} has a nontrivial $p$-superharmonic supersolution then $\sigma$ is absolutely continuous with respect to 
$\text{cap}_{p}(\cdot)$.
%%% ----------------------------------------------------------------------------------------------
%%% ----------------------------------------------------------------------------------------------
\section{Existence of a Positive Finite Energy Solution to Equation \eqref{main_eq_p-lapacain}}

In this section, we establish necessary and sufficient conditions for the existence of a positive finite energy solution to equation \eqref{main_eq_p-lapacain}. Minimality of such a solution is demonstrated as well. In the case $p \geq n$, it follows immediately from the result in \cite{CV14a} that there is only a trivial supersolution to \eqref{main_eq_p-lapacain}. Henceforth, we assume $1<p<n$. 

Our first theorem is stated in the general framework of nonlinear integral equations involving Wolff potentials, 
\begin{equation}\label{int_eq_general}
u = {\W}_{\alpha,p} (u^{q}d\sigma) + {\W}_{\alpha,p}\mu \quad \text{in} \;\; \mathbb{R}^n,
\end{equation}
where $1<p<n$, $0<q<p-1$, $0<\alpha<\frac{n}{p}$ and $\sigma, \mu \in \mathcal{M}^{+}(\mathbb{R}^n)$. This theorem will be used to construct positive finite energy solutions to both equations \eqref{main_eq_p-lapacain} and \eqref{main_eq_frac_laplacian} in the cases $\alpha = 1$ and $p=2$, respectively.

\begin{Thm}\label{thm_existence}
Let $1<p<n$, $0<q<p-1$, $0<\alpha<\frac{n}{p}$ and $\sigma, \mu \in \mathcal{M}^{+}(\mathbb{R}^n)$. Suppose that the following conditions hold:
 \begin{equation}\label{condition_sigma_alpha_p}
 {\W}_{\alpha, p}\sigma \in L^{\frac{(1+q)(p-1)}{p-1-q}}(\R^n, d\sigma), 
 \end{equation}
and 
\begin{equation}\label{condition_f_alpha_p}
{\W}_{\alpha, p}\mu \in L^{1+q}(\R^n, d\sigma).
 \end{equation}
Then there exists a positive solution $u \in L^{1+q}(\mathbb{R}^n, d\sigma)$ to integral equation
\eqref{int_eq_general}.
\end{Thm}

The following result will be used in the proof of Theorem 
\ref{thm_existence} (see \cite[Lemma 3.3]{CV14a}, or \cite{COV06} in more generality).

\begin{Lem} [\cite{CV14a}]\label{bdd_operator} 
Let $1<p<\infty$, $0<q<p-1$, $0<\alpha<\frac{n}{p}$ and $\sigma \in \mathcal{M}^{+}(\mathbb{R}^n)$. Suppose \eqref{condition_sigma_alpha_p} holds. Then the nonlinear integral operator $T$ defined by 
\[
T(g) := \left( \W_{\alpha,p}(|g| d\sigma) \right)^{p-1}, \quad  g \in L^{\frac{1+q}{q}}(\mathbb{R}^n, d\sigma)
\]
is bounded from $L^{\frac{1+q}{q}}(\mathbb{R}^n, d\sigma)$ to $L^{\frac{1+q}{p-1}}(\mathbb{R}^n, d\sigma)$.
\end{Lem}

\begin{proof}[Proof of Theorem \ref{thm_existence}] Without loss of generality we may assume that 
$g \ge 0$, $g \in L^{\frac{1+q}{q}}(\mathbb{R}^n, d\sigma)$. 
Since \eqref{condition_sigma_alpha_p} holds, then by Lemma \ref{bdd_operator}, there exists a positive constant $c$ such that 
\begin{equation}\label{est1_thm_existence}
\left( \int_{\mathbb{R}^n} \big| {\W}_{\alpha, p} (g d\sigma) \big|^{1+q} \;d\sigma \right)^{\frac{1}{1+q}}
\leq c \left( \int_{\mathbb{R}^n} \vert g \vert^{\frac{1+q}{q}}\;d\sigma \right)^{\frac{q}{(1+q)(p-1)}}, 
\end{equation} 
where $c$ is a positive constant that does not depend on $g \in L^{\frac{1+q}{q}}(\mathbb{R}^n, d\sigma)$. We construct a sequence of functions $\lbrace  u_j \rbrace_{0}^{\infty}$ as follows. Set
\[
u_0 := {\W}_{\alpha, p}\mu \qquad \text{and} \qquad u_{j+1} := {\W}_{\alpha, p} (u^{q}_{j} d\sigma) + {\W}_{\alpha, p}\mu, \quad \;\; j \in \mathbb{N}_{0}.
\]
Observe that $u_0 > 0$ since $\mu \not\equiv 0$, and
\[
u_{1} = {\W}_{\alpha, p} (u^{q}_{0} d\sigma) + u_0 \geq u_0.
\]
Suppose $u_0 \leq u_1 \leq \ldots \leq u_j$ for some $j \in \N$. Then
\[
u_{j+1} 
= {\W}_{\alpha, p} (u^{q}_{j} d\sigma) + {\W}_{\alpha, p}\mu 
\geq {\W}_{\alpha, p} (u^{q}_{j-1} d\sigma) + {\W}_{\alpha, p}\mu
= u_{j}.
\]
Hence, by induction, $\lbrace  u_j\rbrace_{0}^{\infty}$ is a nondecreasing sequence of positive functions. Moreover, each $u_j \in  L^{1+q}(\mathbb{R}^n, d\sigma)$. To see this, notice that by assumption \eqref{condition_f_alpha_p}, we have 
\[
u_0 = \W_{\alpha, p}\mu \in L^{1+q}(\mathbb{R}^n, d\sigma).
\]
Suppose $u_0, \ldots, u_j \in L^{1+q}(\mathbb{R}^n, d\sigma)$ for some $j \in \N$. By Minkowski's inequality, 
\begin{equation}\label{est2_thm_existence}
\begin{split}
\Vert u_{j+1} \Vert_{L^{1+q}(\mathbb{R}^n,\; d\sigma)}
&= \big\Vert \W_{\alpha, p} (u^{q}_{j} d\sigma) + \W_{\alpha, p}\mu \big\Vert_{L^{1+q}(\mathbb{R}^n, \; d\sigma)} \\
&\leq \big\Vert \W_{\alpha, p} (u^{q}_{j} d\sigma) \big\Vert_{L^{1+q}(\mathbb{R}^n, \; d\sigma)} + \big\Vert \W_{\alpha, p}\mu  \big\Vert_{L^{1+q}(\mathbb{R}^n, \; d\sigma)}.
\end{split}
\end{equation}
The first term on the right-hand side of \eqref{est2_thm_existence} is estimated by applying \eqref{est1_thm_existence} with $g:=u^{q}_{j} \in 
L^{\frac{1+q}{q}}(\mathbb{R}^n, d\sigma)$. In fact,
\begin{equation}\label{est3_thm_existence}
\begin{split}
\big\Vert \W_{\alpha, p} (u^{q}_{j} d\sigma) \big\Vert_{L^{1+q}(\mathbb{R}^n, d\sigma)}
&\leq c \left( \int_{\mathbb{R}^n} u^{1+q}_{j}\;d\sigma \right)^{\frac{q}{(1+q)(p-1)}} \\
&\leq c \left( \int_{\mathbb{R}^n} u^{1+q}_{j+1}\;d\sigma \right)^{\frac{q}{(1+q)(p-1)}} \\
&= c \Vert u_{j+1}\Vert^{\frac{q}{p-1}}_{L^{1+q}(\mathbb{R}^n, d\sigma)}.
\end{split}
\end{equation}
Combining \eqref{est2_thm_existence} and \eqref{est3_thm_existence}, we arrive at 
\begin{equation} \label{est4_thm_existence}
\Vert u_{j+1} \Vert_{L^{1+q}(\mathbb{R}^n, d\sigma)}
\leq c \Vert u_{j+1}\Vert^{\frac{q}{p-1}}_{L^{1+q}(\mathbb{R}^n, d\sigma)} + \big\Vert \W_{\alpha, p}\mu  \big\Vert_{L^{1+q}(\mathbb{R}^n, d\sigma)}.
\end{equation}
We estimate the first term on the right-hand side of \eqref{est4_thm_existence} using Young's inequality,
\begin{equation}\label{est5_thm_existence}
c \Vert u_{j+1}\Vert^{\frac{q}{p-1}}_{L^{1+q}(\mathbb{R}^n, d\sigma)} 
\leq \frac{q}{p-1} \Vert u_{j+1}\Vert_{L^{1+q}(\mathbb{R}^n, d\sigma)} +  \frac{p-1-q}{p-1}c^{\frac{p-1}{p-1-q}}
\end{equation}
Hence, by \eqref{est4_thm_existence} and \eqref{est5_thm_existence}, we obtain
\begin{equation}\label{est6_thm_existence}
\Vert u_{j+1} \Vert_{L^{1+q}(\mathbb{R}^n, d\sigma)} 
\leq c^{\frac{p-1}{p-1-q}} + \frac{p-1}{p-1-q} \big\Vert \W_{\alpha, p}\mu
\big\Vert_{L^{1+q}(\mathbb{R}^n, d\sigma)} < +\infty.
\end{equation}
By induction, we have shown that each $u_j \in  L^{1+q}(\mathbb{R}^n, d\sigma)$. Finally, applying the Monotone Convergence Theorem to the sequence $\lbrace u_{j} \rbrace_{0}^{\infty}$, we see that the pointwise limit  
\[
u:=\lim_{j \rightarrow \infty} u_j
\] 
exists so that $u>0$, $u \in L^{1+q}(\mathbb{R}^n, d\sigma)$ and satisfies \eqref{int_eq_general}.
\end{proof}

\begin{Rem}\label{rem_existence}
The converse to Theorem \ref{thm_existence} is also true in a more general sense. In fact, if  $u \in L^{1+q}(\R^n, d\sigma)$, $u > 0 $ $d\sigma$-a.e.,  satisfies the equation 
\[
u = \W_{\alpha,p} (u^{q}d\sigma) + \W_{\alpha,p}\mu \quad d\sigma\text{-}a.e.,
\]
then obviously $u \in L^{q}_{loc}(\mathbb{R}^n, d\sigma)$ by H\"{o}lder's inequality, and 
\[
u \geq  \W_{\alpha, p} (u^{q}d\sigma) \quad d\sigma\text{-}a.e.
\] 
Applying Theorem \ref{lower_ptwise_est}, we obtain a lower pointwise estimate of  $u$, 
\[
u \geq c \left( \W_{\alpha, p}\sigma \right)^{\frac{p-1}{p-1-q}} \quad d\sigma\text{-}a.e.,
\]
where $c = c(\alpha, n, p, q) > 0$. This implies that
\eqref{condition_sigma_alpha_p} holds since $u \in L^{1+q}(\R^n, d\sigma)$. Similary, \eqref{condition_f_alpha_p} holds because $u \in L^{1+q}(\R^n, d\sigma)$ and 
\[
u \geq \W_{\alpha, p}\mu \qquad d\sigma\text{-}a.e.
\]
\end{Rem}

The next lemma is our main observation in this section. It gives us a relation 
between conditions \eqref{condition_sigma_alpha=1}, \eqref{condition_mu_alpha=1} and  \eqref{condition_mu_1+q}.

\begin{Lem} \label{relation_condition_p-laplacian}
Let $1<p<n$, $0<q<p-1$ and $\sigma, \mu \in \mathcal{M}^{+}(\mathbb{R}^n)$. Then conditions \eqref{condition_sigma_alpha=1} and \eqref{condition_mu_alpha=1} imply \eqref{condition_mu_1+q}.
\end{Lem}

\begin{proof}
As shown in \cite{COV00}, \eqref{condition_sigma_alpha=1}
holds if and only if there exists a positive constant $c$ such that
\begin{equation}\label{condition_COV00_4.1}
\big\Vert \I_{1}g \big\Vert_{L^{1+q}(\R^n, d\sigma)}
\leq c \Vert g \Vert_{L^{p}(\R^n)}, \quad \forall g \in L^{p}(\R^n).
\end{equation}
Since $\mu \in L^{-1,p'}(\R^n)$ then $\I_{1}\mu \in L^{p'}(\R^n)$. Substituting $g := \left( \I_{1}\mu \right)^{\frac{1}{p-1}}\in L^{p}(\R^n)$ into \eqref{condition_COV00_4.1} yields  
\[
\big\Vert \I_{1} \left( \I_{1}\mu \right)^{\frac{1}{p-1}} \big\Vert_{L^{1+q}(\R^n, d\sigma)}
\leq c \big\Vert \left( \I_{1}\mu \right)^{\frac{1}{p-1}} \big\Vert_{L^{p}(\R^n)} < +\infty.
\] 
Notice that $\W_{1,p}\mu$ is always pointwise smaller than $\I_{1}\left( \I_{1}\mu \right)^{\frac{1}{p-1}} $ (see, for example, \cite[Sec. 10.4.2]{Maz11}). More precisely, 
\[
\W_{1,p}\mu  \leq C \I_{1}\left( \I_{1}\mu \right)^{\frac{1}{p-1}},
\]
where $C$ is a constant which depends only on $p$. This yields \eqref{condition_mu_1+q}.
\end{proof}

The following lemma, in particular, gives necessary conditions for the existence of 
a positive finite energy solution to equation \eqref{main_eq_p-lapacain}.

\begin{Lem} \label{converse_existence_thm_inhomo} 
Let $1<p<n$, $0<q<p-1$ and $\sigma, \mu \in \mathcal{M}^{+}(\mathbb{R}^n)$.
Suppose there exists a nontrivial supersolution $u \in L_{loc}^{q}(\R^n, d\sigma) \cap \dot{W}_{0}^{1,p}(\R^n)$ to equation
\eqref{main_eq_p-lapacain}.
Then
\[
-\Delta_{p} u \in W^{-1,p'}(\R^n) \cap \mathcal{M}^{+}(\R^n)
\quad 
\text{and}
\quad
u \in L^{1+q}(\mathbb{R}^n, d\sigma)
\] 
for a quasicontinuous representative of $u$. Consequently, \eqref{condition_sigma_alpha=1} and \eqref{condition_mu_alpha=1} hold.
\end{Lem}

\begin{proof}
It follows directly from Lemma \ref{converse_existence_thm} that  
\[
-\Delta_{p} u \in W^{-1,p'}(\mathbb{R}^n) \cap \mathcal{M}^{+}(\mathbb{R}^n)
\quad 
\text{and}
\quad
u \in L^{1+q}(\mathbb{R}^n, d\sigma)
\] 
for a quasicontinuous representative of $u$. The former implies \eqref{condition_mu_alpha=1}. The latter yields \eqref{condition_sigma_alpha=1} in view of the global pointwise lower bound  for supersolutions contained (Theorem \ref{lower_ptwise_est_homo}).
\end{proof}

In the next theorem, we verify that conditions \eqref{condition_sigma_alpha=1} and \eqref{condition_mu_alpha=1} are sufficient for the existstence of 
a positive finite energy solution to equation \eqref{main_eq_p-lapacain}. Further, the minimality of such a solution is also proven. 

We first observe that for $1<p<n$, $0<\alpha<\frac{n}{p}$, $\omega, \nu \in \mathcal{M}^{+}(\R^n)$ and $\gamma, \beta \geq 0$, 
\begin{equation}\label{inq}
\W_{\alpha, p}(\gamma\omega + \beta\nu) \leq A \left( \gamma^{\frac{1}{p-1}} \W_{\alpha, p}\omega + \beta^{\frac{1}{p-1}} \W_{\alpha, p}\nu\right),
\end{equation}
where $A = A(\alpha, p, n) \geq 1 $. This follows immediately from the definition of Wolff's potential and the estimates 
\[
|a+b|^{r} \leq 
\begin{cases}
2^{r-1} \left( |a|^{r} + |b|^{r} \right) \quad \text{for} \;\; 1\leq r < \infty, \\
 |a|^{r} + |b|^{r}  \qquad \quad \;\;\; \text{for} \;\; 0 \leq r < 1,
\end{cases}
\]
where $a,b \in \R$.
\begin{Thm} \label{thm_main}
Let $1<p<n$, $0<q<p-1$ and $\sigma, \mu \in \mathcal{M}^{+}(\mathbb{R}^n)$. Suppose
\eqref{condition_sigma_alpha=1} and \eqref{condition_mu_alpha=1} hold. Then there exists a positive finite energy solution $w$ to equation \eqref{main_eq_p-lapacain}. Moreover, $w$ is a minimal solution in the sense that $w \leq u$ q.e. (for their respective quasicontinuous representatives) for any positive finite energy solution $u$ to \eqref{main_eq_p-lapacain}.
\end{Thm}

\begin{proof}
We first prove the existence of $w$. Since \eqref{condition_sigma_alpha=1} and \eqref{condition_mu_alpha=1} hold, then by Lemma \ref{relation_condition_p-laplacian},  \eqref{condition_mu_1+q} holds. By Theorem \ref{thm_existence} in the case $\alpha =1$,
there exists a positive solution $v \in L^{1+q}(\R^n, d\sigma)$ to the integral equation 
\[
v = \W_{1,p}(v^{q}d\sigma) + \W_{1,p}\mu \qquad \text{in} \;\; \R^n.
\]
Using a constant multiple $c^{-1}v$, where $c>0$, in place of $v$, we have
\[
v = c^{\frac{p-1-q}{p-1}} \W_{1,p}(v^{q}d\sigma) + c\W_{1,p}\mu \qquad \text{in}\;\; \R^n.
\]
Choose $c \geq (KA)^{\frac{p-1}{p-1-q}} \geq  KA \geq 1$ where $K \geq 1$ is the constant in Theorem \ref{pointwise_est_p-superharmonic}, and $A \geq 1$ is the constant in \eqref{inq}. Then, by Lemma \ref{regulartiy_of_measure}, we have
\[
v^{q}d\sigma \in W^{-1,p'}(\R^n).
\]
Set 
\[
w_{0} :=  K^{-1} \W_{1,p}\mu \quad \text{and} \quad d\omega_{0}:={w_{0}}^{q}d\sigma + \mu.
\]
Since $K^{-1} \leq 1 \leq c $ then $0 < w_0 \leq v$, and hence  
\[
w_{0} \in L^{1+q}(\R^n, d\sigma) \quad \text{and} \quad \omega_{0} \in W^{-1, p'}(\R^n).
\]
As discussed in Remark \ref{Rem0}, for such a measure $\omega_{0}$, there exists a unique $p$-superharmonic solution 
$w_1 \in \dot{W}^{1,p}_{0}(\R^n)$ to the equation
\[
-\Delta_{p}w_{1} = \omega_{0} \quad \text{in} \;\; \R^n \quad \text{and} \quad \Vert w_1 \Vert_{\dot{W}^{1,p}_{0}(\R^n)}^{p-1} = \Vert \omega_{0} \Vert_{W^{-1,p}(\R^n)}
\]
for a quasicontinuous representative of $w_1$. Moreover, by Theorem \ref{pointwise_est_p-superharmonic},
\[
0 < w_1 \leq K \W_{1,p}\omega_{0} \leq KA \W_{1,p}(v^{q}d\sigma) + KA \W_{1,p}f \leq v \quad \text{q.e.}
\]
Since $\sigma$ is absolutely continuous with respect to $\text{cap}_{p}(\cdot)$, this yields
\[
w_1 \in L^{1+q}(\R^n, d\sigma) \quad \text{and} \quad  \omega_{1}:={w_{1}}^{q}d\sigma + d\mu  \in W^{-1, p'}(\R^n).
\] 
Again by Theorem \ref{pointwise_est_p-superharmonic},
\[
 w_{0} = K^{-1} \W_{1,p}\mu \leq K^{-1} \W_{1,p}\omega_{0} \leq w_{1} 
\quad \text{q.e.}
\]
We now have
\[
0 < w_0 \leq w_1 \leq v \quad \text{q.e.}
\]
We shall construct, by induction, a sequence $ \lbrace w_j \rbrace_{1}^{\infty}$ so that 
\begin{equation}\label{sequence_construction}
\begin{cases}
-\Delta_{p} w_{j} = \sigma w^{q}_{j-1} + \mu \quad \text{in} \;\; \mathbb{R}^n, \\ 
w_{j} \in  L^{1+q}(\R^n, d\sigma) \cap \dot{W}^{1,p}_{0}(\R^n), \\ 
\sup_{j \in \N} \Vert w_j \Vert_{\dot{W}_{0}^{1,p}(\R^n)} < \infty, \\
w_{j-1}^{q}d\sigma + d\mu \in W^{-1,p'}(\R^n), \\
0 < w_{j-1} \leq w_{j} \leq v \;\; \text{q.e.}
\end{cases}
\end{equation}
We set 
\[
d\omega_{j} := w^{q}_{j}d\sigma + d\mu, \quad j \in \N.
\] 
Suppose $ w_{1}, w_{2},..., w _{j-1}$ have been constructed. Since $\omega_{j-1} \in W^{-1,p'}(\R^n)$, then by Remark \ref{Rem0}, there exists a unique $p$-superharmonic solution $w_{j} \in \dot{W}^{1,p}_{0}(\R^n)$ to the equation
\[
-\Delta_{p} w_{j} = \omega_{j-1} \quad \text{in} \;\; \R^n.
\]
Moreover,
\[
\Vert w_{j} \Vert_{\dot{W}^{1,p}_{0}(\R^n)}^{p} = \Vert \omega_{j-1} \Vert_{W^{-1,p'}(\R^n)}^{p'} \leq 
\int_{\R^n} w_{j}w_{j-1}^{q}\;d\sigma + \Vert \mu \Vert_{W^{-1,p'}(\R^n)}^{p'}.
\]
Applying Theorem \ref{pointwise_est_p-superharmonic}, we obtain
\[
w_{j} \leq K \W_{1,p} \omega_{j-1} \leq KA \W_{1,p}(w_{j-1}^{q}d\sigma) + KA \W_{1,p}\mu \quad \text{q.e.}
\]
Since $w_{j-1} \leq v$ q.e. then
\[
w_{j} \leq KA \W_{1,p}(v^{q}d\sigma) + KA \W_{1,p}\mu \leq  v \quad \text{q.e.}
\]
Hence, $w_{j} \in L^{1+q}(\R^n, d\sigma)$ since $\sigma$ is absolutely continuous with respect to $\text{cap}_{p}(\cdot)$. Furthermore,
\begin{align*}
\Vert w_{j} \Vert_{\dot{W}_{0}^{1,p}(\R^n)}^{p} 
& \leq \int_{\R^n} w_{j}w_{j-1}^{q}\;d\sigma + \Vert \mu \Vert_{W^{-1,p'}(\R^n)}^{p'}\\
& \leq \int_{\R^n} v^{1+q}\;d\sigma + \Vert \mu \Vert_{W^{-1,p'}(\R^n)}^{p'} < +\infty.
\end{align*}
This shows that $\lbrace w_{j} \rbrace_{1}^{\infty}$ is a bounded sequence in $\dot{W}^{1,p}_{0}(\R^n)$.
Moreover, since $\omega_{j-2} \leq \omega_{j-1}$, then by Weak Comparison Principle (Lemma \ref{weak_comparison_principle}), $w_{j-1} \leq w_{j}$ q.e. Hence, the sequence  $\lbrace w_{j} \rbrace_{1}^{\infty}$ satisfying 
\eqref{sequence_construction} has been constructed. Applying the weak continuity 
of $p$-Laplacian (Theorem \ref{weak_cont_p-Laplacian}), the Monotone Convergence Theorem and  the Weak Compactness Property in $\dot{W}_{0}^{1,p}(\R^n)$, see
 \cite[Lemma 1.33]{HKM06}, we deduce that the pointwise limit 
$w := \lim_{j \rightarrow \infty} w_j$ is a positive finite energy solution to \eqref{main_eq_p-lapacain}.

We now prove the minimality of $w$. Suppose $u$ is any positive finite energy solution to \eqref{main_eq_p-lapacain}. Set $d\omega := u^{q}d\sigma + d\mu$. By Lemma \ref{converse_existence_thm_inhomo}, we have 
\[
u \in L^{1+q}(\R^n, d\sigma) \quad \text{and} \quad \omega \in W^{-1,p'}(\R^n) \cap \mathcal{M}^{+}(\R^n),
\]
for a quasicontinuous representative of $u$. We need to show that $w \leq u$ q.e. Notice that
\[
u \geq (\W_{1,p}\mu) > K^{-1}(\W_{1,p}\mu) = w_0 \quad \text{q.e.}
\]
Therefore $\omega_{0} \leq \omega $ since $\sigma$ is absolutely continuous with respect to $\text{cap}_{p}(\cdot)$. By the Weak Comparison Principle (Lemma \ref{weak_comparison_principle}), $w_1 \leq u$ q.e. Arguing by induction as above, we see that 
\[
w_{j-1} \leq w_{j} \leq u \quad \text{q.e.}
\]
It follows that $ w =\lim_{j \rightarrow \infty} w_{j} \leq u$\, q.e., which proves the claim.
\end{proof}

\begin{Rem}\label{phuc-v}
For a similar equation in a domain $\Omega \subset \R^n$,
\begin{equation}
-\Delta_{p}u = u^{q}\sigma + \mu \quad \text{in} \;\; \Omega,
\end{equation}
where $1<p<n$, $0<q<p-1$ and $\sigma, \mu \in \mathcal{M}^{+}(\Omega)$, we also have analogous sufficient conditions for the existence of a positive finite energy solution in terms of truncated Wolff's potential, namely:
\begin{equation}\label{condition_sigma_alpha=1_domain}
\W_{1, p}^{R}\sigma \in L^{\frac{(1+q)(p-1)}{p-1-q}}(\Omega, d\sigma), \quad R \geq 2\text{diam}(\Omega),
\end{equation}
and 
\begin{equation}\label{condition_mu_alpha=1_domain}
\mu \in \dot{W}^{-1,p'}(\Omega).
\end{equation}
Here, for $1<p<\infty$, $0< \alpha < \frac{n}{p}$ and $\sigma \in \mathcal{M}^{+}(\Omega)$, the truncated Wolff potential $\W_{\alpha, p}^{R}\sigma$  is defined by (see \cite{KuMi14})
\[
\W_{\alpha, p}^R\sigma(x) = \int_{0}^{R} \left[ \frac{\sigma(B(x,r)\cap\Omega)}{r^{n-\alpha p}}\right]^{\frac{1}{p-1}}\; \frac{dr}{r}, \quad x \in \Omega, 
0 < R \leq +\infty.
\]
Moreover, conditions \eqref{condition_sigma_alpha=1_domain} and \eqref{condition_mu_alpha=1_domain} are also necessary whenever $\sigma$ and $\mu$ have compact supports in $\Omega$. These results are deduced easily from Theorem~\ref{thm_main1}; see details in \cite{PV08}. 
\end{Rem}
%%% ----------------------------------------------------------------------------------------------
%%% ----------------------------------------------------------------------------------------------
\section{Existence of a Positive Finite Energy Solution to Equation \eqref{main_eq_frac_laplacian}}

In this section, we employ an argument similar to the one used  in  the previous section to deduce necessary and sufficient conditions for the existence of a positive finite energy solution to fractional Laplace equation \eqref{main_eq_frac_laplacian}.

\begin{Def} \label{Def_Sol_2}
Let $0<q<1$, $0< \alpha < \frac{n}{2}$ and $\sigma, \mu \in \mathcal{M}^{+}(\mathbb{R}^n)$. A finite energy solution $u$ to equation
\eqref{main_eq_frac_laplacian} will be understood in the sense that
$u \in L^{q}_{loc}(\R^n, d\sigma) \cap \dot{H}^{\alpha}(\R^n)$, $u \geq 0\;\;d\sigma\text{-}a.e.$ such that 
\begin{equation} \label{int_eq_frac_laplacian}
\left( -\Delta \right)^{\frac{\alpha}{2}}u = \I_{\alpha} (u^{q}d\sigma) + \I_{\alpha}\mu  \qquad dx\text{-}a.e.
\end{equation}
\end{Def}

\begin{Rem}\label{Rem_Def_2}
Using the same notation as above, suppose $u$ is a positive   finite energy solution to \eqref{main_eq_frac_laplacian}. Applying the Riesz potential $\I_{\alpha}$ of order $\alpha$ to both sides of \eqref{int_eq_frac_laplacian} yields
\[
u(x)= \I_{2\alpha}(u^{q}d\sigma)(x) + \I_{2\alpha}\mu(x)
\quad \text{whenever} \;\; u(x) < +\infty.
\]
Notice that  $u \in L^{q}_{loc}(\R^n, d\sigma) \cap \dot{H}^{\alpha}(\R^n)$. Then 
\begin{equation}\label{eq_int}
u = \I_{2\alpha}(u^{q}d\sigma) + \I_{2\alpha}\mu \quad d\sigma\text{-}a.e. \;\;\text{and}\;\; q.e.
\end{equation}
In particular,
\[
u \geq \I_{2\alpha}(u^{q}d\sigma) \quad d\sigma\text{-}a.e.,
\]
which implies, by Lemma \ref{lemma_cap}, that  $\sigma$ is absolutely continuous with respect to $\text{cap}_{\alpha, 2}(\cdot)$. On the other hand, \eqref{Def_Sol_2} implies in particular that 
\[
\left( -\Delta \right)^{\frac{\alpha}{2}}u  \geq \I_{\alpha}\mu  \quad dx\text{-}a.e.
\]
Therefore $\I_{\alpha}\mu \in L^{2}(\R^n)$, and hence $\mu \in \dot{H}^{-\alpha}(\R^n)$. 
In particular, $\mu$ is absolutely continuous with respect to $\text{cap}_{\alpha, 2}(\cdot)$ 
(see, for example, \cite[Sec. 7]{AH96}). In summary, $u$ satisfies integral equation \eqref{eq_int} in the following senses: a.e., $d\sigma$-a.e., $d\mu$-a.e., and q.e.
\end{Rem}

The following important observation is analogous to Lemma \ref{relation_condition_p-laplacian}.

\begin{Lem} \label{relation_of_conditions}
Let $0<q<1$, $0< \alpha < \frac{n}{2}$ and $\sigma, \mu \in \mathcal{M}^{+}(\mathbb{R}^n)$. Then \eqref{condition_sigma_alpha_2} and \eqref{condition_mu_alpha_2} imply 
\eqref{condition_f_alpha_2}
\end{Lem}
\begin{proof}
As shown in \cite{COV00}, \eqref{condition_sigma_alpha_2}
holds if and only if there exists a positive constant $c$ such that
\begin{equation}\label{condition_COV00_1}
\big\Vert \I_{\alpha}g \big\Vert_{L^{1+q}(\R^n, d\sigma)}
\leq c \Vert g \Vert_{L^{2}(\R^n)}, \quad \forall g \in L^{2}(\R^n).
\end{equation}
Letting $g := \I_{\alpha}\mu \in L^{2}(\R^n)$ in \eqref{condition_COV00_1}, we have 
\[
\big\Vert \I_{2\alpha}\mu \big\Vert_{L^{1+q}(\R^n, d\sigma)}
\leq c \big\Vert \I_{\alpha}\mu \big\Vert_{L^{2}(\R^n)} < \infty,
\] 
which proves \eqref{condition_f_alpha_2}.
\end{proof}

The neccessary conditions for the existence of a positive finite energy solution to equation \eqref{main_eq_frac_laplacian} are established in the following lemma.

\begin{Lem} \label{lemma_converse_2}
Let $0<q<1$, $0< \alpha < \frac{n}{2}$ and $\sigma, \mu \in \mathcal{M}^{+}(\mathbb{R}^n)$. Suppose there exists a positive finite energy solution $u$ to equation \eqref{main_eq_frac_laplacian}. Then \eqref{condition_mu_alpha_2} holds and $u \in L^{1+q}(\R^n, d\sigma)$. Consequently, \eqref{condition_sigma_alpha_2} holds.
\end{Lem}

\begin{proof}
Suppose $u$ is a positive finite energy solution to \eqref{main_eq_frac_laplacian}. Then
\eqref{condition_mu_alpha_2} holds as discussed in Remark \ref{Rem_Def_2}. We next show that $u \in L^{1+q}(\R^n, d\sigma)$. 
By \eqref{int_eq_frac_laplacian}, for each nonnegative function $\varphi \in L^{2}(\R^n)$, we have
\[
\int_{\R^n} [ (-\Delta)^{\frac{\alpha}{2}}u ] \varphi \;dx = 
\int_{\R^n} [\I_{\alpha} (u^{q}d\sigma)] \varphi \;dx  + 
\int_{\R^n} [\I_{\alpha}\mu] \varphi \;dx.
\]
Applying Tonelli's Theorem and Schwarz's inequality, we obtain
\begin{equation}\label{converse2_est1}
\begin{split}
\bigg\vert \int_{\R^n} u^{q} \left[ \I_{\alpha}\varphi \right]\;d\sigma \bigg\vert 
&= \bigg\vert \int_{\R^n} \left[\I_{\alpha}(u^q d\sigma) \right] \varphi \;dx \bigg\vert   \\
&\leq  \bigg\vert \int_{\R^n} \left[ (-\Delta)^{\frac{\alpha}{2}}u \right] \varphi \;dx \bigg\vert
+   \bigg\vert \int_{\R^n} (\I_{\alpha}\mu) \varphi \;dx \bigg\vert   \\
&\leq c \Vert \varphi \Vert_{L^{2}(\R^n)},
\end{split}
\end{equation}
where
$
c:= \Vert (-\Delta)^{\frac{\alpha}{2}}u\Vert_{L^{2}(\R^n)} + 
\Vert \I_{\alpha}\mu \Vert_{L^{2}(\R^n)}<\infty
$, since $u \in \dot{H}^{\alpha}(\R^n)$ and $\mu \in \dot{H}^{-\alpha}(\R^n)$.
Letting $\varphi := (-\Delta)^{\frac{\alpha}{2}}u$, which is a nonnegative function of class $L^{2}(\R^n)$ in \eqref{converse2_est1}, we get
\[
\Vert u \Vert_{L^{1+q}(\R^n, d\sigma)}^{1+q} \leq c \big\Vert (-\Delta)^{\frac{\alpha}{2}}u \big\Vert_{L^2(\R^n)} < +\infty.
\]
This shows that $u \in L^{1+q}(\R^n, d\sigma)$. Notice that
\[
u = \I_{2\alpha}(u^{q}d\sigma) + \I_{2\alpha}\mu \quad d\sigma\text{-}a.e.
\]
Hence, by the discussion in Remark \ref{rem_existence} in the case $p=2$, we have that \eqref{condition_sigma_alpha_2} holds.
\end{proof}

The next theorem shows that conditions \eqref{condition_sigma_alpha_2} and \eqref{condition_mu_alpha_2}  allow us to construct a positive finite energy solution to equation \eqref{main_eq_frac_laplacian}. Minimality of such a solution will be proven as well.

\begin{Thm} \label{thm_converse_1}
Let $0<q<1$, $0< \alpha < \frac{n}{2}$ and $\sigma, \mu \in \mathcal{M}^{+}(\mathbb{R}^n)$. Suppose \eqref{condition_sigma_alpha_2} and \eqref{condition_mu_alpha_2} hold. Then there exists a positive finite energy solution $w$ to equation \eqref{main_eq_frac_laplacian}.
Moreover, $w$ is a minimal solution in the sense that $w \leq u$ q.e. for any positive finite energy solution $u$ to
\eqref{main_eq_frac_laplacian}.
\end{Thm}

\begin{proof}
We first prove the existence of $w$. Since \eqref{condition_sigma_alpha_2} and \eqref{condition_mu_alpha_2} hold, then by Lemma \ref{relation_of_conditions}  it follows that  
\eqref{condition_f_alpha_2} holds. By Theorem \ref{thm_existence} in the case $p=2$, there exists a positive solution $w \in L^{1+q}(\R^n, d\sigma)$ to the integral equation 
\begin{equation}\label{int_eq_p=2}
w = \I_{2\alpha} (w^{q}d\sigma) + \I_{2\alpha}\mu \quad \text{in} \;\; \mathbb{R}^n.
\end{equation}
We will show that
\[
w \in L^{q}_{loc}(\R^n, d\sigma) \cap \dot{H}^{\alpha}(\R^n).
\]
Clearly, $w \in L^{q}_{loc}(\R^n, d\sigma)$ by H{\"o}lder's inequality.
In order to prove that $w \in \dot{H}^{\alpha}(\R^n)$, by duality, it suffices to show that there exists a positive constant $c$ such that 
\begin{equation}\label{est1_converse1}
\bigg\vert \int_{\R^n} w \psi \;dx \bigg\vert  
\leq c \Vert \psi \Vert_{\dot{H}^{-\alpha}(\R^n)},
\quad \psi \in C_{0}^{\infty}(\R^n).
\end{equation}
By the semigroup property of the Riesz potentials, 
Tonelli's Theorem and H{\"o}lder's inequality, we have 
\begin{equation}\label{est2_converse1}
\begin{split}
\bigg\vert \int_{\R^n} w\psi \;dx \bigg\vert 
&\leq \int_{\R^n}  \I_{\alpha}(w^q d\sigma)   \big| \I_{\alpha}\psi \big| \;dx 
+ \int_{\R^n} \I_{\alpha}\mu  \big| \I_{\alpha}\psi \big| \;dx \\
&\leq  \big\Vert \I_{\alpha}(w^q d\sigma) \big\Vert_{L^{2}(\R^n)}
\big\Vert \I_{\alpha}\psi \big\Vert_{L^{2}(\R^n)}
+ \big\Vert \I_{\alpha}\mu \big\Vert_{L^{2}(\R^n)}
\big\Vert \I_{\alpha}\psi \big\Vert_{L^{2}(\R^n)} \\
&= \left[ \big\Vert \I_{\alpha}(w^q d\sigma) \big\Vert_{L^{2}(\R^n)}
+ \Vert \mu \Vert_{\dot{H}^{-\alpha}(\R^n)} \right] \Vert \psi \Vert_{\dot{H}^{-\alpha}(\R^n)}
 \\
\end{split}
\end{equation}
for all $\psi \in C_{0}^{\infty}(\R^n)$. Since $\Vert \mu \Vert_{\dot{H}^{-\alpha}(\R^n)} < +\infty$, 
we see that, in view of \eqref{est1_converse1} and \eqref{est2_converse1}, it remains to show that
\begin{equation}\label{toshow}
\big\Vert \I_{\alpha}(w^q d\sigma) \big\Vert_{L^{2}(\R^n)} < \infty.
\end{equation}
To this end, notice that by the result in \cite{COV00}, \eqref{condition_sigma_alpha_2} is equivalent to 
\begin{equation}\label{condition_COV00_2}
\Vert \I_{\alpha}g \Vert_{L^{1+q}(\R^n, d\sigma)}
\leq c \Vert g \Vert_{L^{2}(\R^n)}, \quad \forall g \in L^{2}(\R^n),
\end{equation}
where $c$ is a positive constant independent of $g$. Moreover, by duality, \eqref{condition_COV00_2} is equivalent to
\begin{equation}\label{duality}
\Vert \I_{\alpha}(\varphi d\sigma) \Vert_{L^{2}(\R^n)}
\leq c \Vert \varphi \Vert_{L^{\frac{1+q}{q}}(\R^n, d\sigma)}, 
\quad \forall \varphi \in L^{\frac{1+q}{q}}(\R^n, d\sigma),
\end{equation}
where $c$ is a positive constant independent of $\varphi$. Letting $\varphi := w^q \in L^{\frac{1+q}{q}}(\R^n, d\sigma)$ in \eqref{duality}, we have 
\begin{equation}\label{converse1_est3}
 \big\Vert \I_{\alpha}(w^{q} d\sigma) \big\Vert_{L^{2}(\R^n)} 
\leq c \Vert w \Vert_{L^{1+q}(\R^n, d\sigma)}^{q} < +\infty,
\end{equation}
which proves \eqref{toshow}, and hence $w \in \dot{H}^{\alpha}(\R^n)$. Moreover, by \eqref{int_eq_p=2}, we have
\[
\left( - \Delta \right)^{\frac{\alpha}{2}}w = \I_{\alpha} (w^{q}d\sigma) + \I_{\alpha}\mu \quad a.e.
\]
This shows that $w$ is a positive finite energy solution to \eqref{main_eq_frac_laplacian}.

Minimality of the solution $w$ is obvious by its construction in Theorem \ref{thm_existence} in the case $p=2$. Recall that $w$ is the pointwise limit 
$
w = \lim_{j \rightarrow \infty} w_{j},
$
where 
\[
w_0 := {\I}_{2\alpha}\mu \qquad \text{and} \qquad w_{j+1} := {\I}_{2\alpha} (w^{q}_{j} d\sigma) + {\I}_{2\alpha}\mu, \quad \;\; j \in \mathbb{N}_{0}.
\]
If $u$ is any positive finite energy solution to \eqref{main_eq_frac_laplacian}. Then 
\[
w_{0} = \I_{2\alpha}\mu \leq \I_{2\alpha}(u^{q}d\sigma) + \I_{2\alpha}\mu = u \quad q.e.
\]
Consequently, 
\[
w_{1} =  \I_{2\alpha}(w_{0}^{q}d\sigma) + \I_{2\alpha}\mu \leq  \I_{2\alpha}(u^{q}d\sigma) + \I_{2\alpha}\mu = u \quad q.e.
\]
Arguing by induction, we obtain
\[
w_{j-1} \leq w_{j} \leq u \quad q.e. \quad \text{for all} \;\; j \in \N.
\]
Therefore, $ w = \lim_{j \rightarrow \infty} w_{j}  \leq u \;\; q.e.$ This proves the minimality of $w$.
\end{proof}
%%% ----------------------------------------------------------------------------------------------
%%% ----------------------------------------------------------------------------------------------
\section{Existence of a Positive Finite Energy Solution to Equation \eqref{main_eq_frac_laplacian_domain}}
Let $\Omega$ be a domain in $\R^n$ and let $G: \Omega \times \Omega \rightarrow (0, \infty]$ be a positive lower semicontinuous kernel. For $\nu \in \mathcal{M}^{+}(\Omega)$, the potential of $\nu$ is defined by 
\[
\g \nu (x) := \int_{\Omega} G(x,y) \;d\nu (y), \quad x \in \Omega.
\]

A positive kernel $G$ on $\Omega \times \Omega$ is said to satisfy the {\it weak maximum principle} (WMP) with constant $h \geq 1$ if for any $\nu \in \mathcal{M}^{+}(\Omega)$,
\begin{equation}\label{WMP}
\sup \{ \g\nu(x) : x \in \text{supp}(\nu) \} \leq M 
\Longrightarrow
\sup \{ \g\nu(x) : x \in \Omega \} \leq hM 
\end{equation}
for every constant $M > 0$. Here we use the notation supp$(\nu)$ for the support of $\nu \in \mathcal{M}^{+}(\Omega)$. 

When $h=1$ in \eqref{WMP}, the positive kernel $G$ is said to satisfy the {\it strong maximum principle}, which holds for positive Green's functions associated with the classical Laplacian $-\Delta$, and more generally the fractional Laplacian $(-\Delta)^{\alpha}$ in the case $0<\alpha\leq 1$, for every domain $\Omega \subset \R^n$ which possesses a positive Green's function. 

The WMP holds for Riesz kernels on $\R^n$ associated with $(-\Delta)^{\alpha}$ in the full range $0< \alpha < \frac{n}{2}$, and more generally for all radially nonincreasing kernels on $\R^n$ (see \cite{AH96}).

We say that a function $d(x,y): \Omega \times \Omega \rightarrow [0, \infty)$ satisfies the {\it quasimetric triangle inequality} with constant $\kappa >0$ if
\begin{equation}\label{quasimetric}
d(x,y) \leq \kappa \left[ d(x,z) + d(z,y) \right], \quad x, y, z \in \Omega.
\end{equation}

A positive kernel $G$ on $\Omega \times \Omega$ is {\it quasimetric} if $G$ is symmetric and the function $d(x,y) = \frac{1}{G(x,y)}$ satisfies \eqref{quasimetric}. The WMP holds for quasimetric kernels, see \cite{FNV14, FV, HN12, QV17}. We say that a positive kernel $G$ on $\Omega \times \Omega$ is {\it quasi-symmetric} if there exists a constant $a>0$ such that 
\[
a^{-1}G(y,x) \leq G(x,y) \leq a G(y,x), \quad x,y \in \Omega.
\]
There are many kernels associated with elliptic operators that are quasi-symmetric and satisfy the WMP (see \cite{Anc02}).

In this section, we establish necessary and sufficient conditions for the existence of a positive solution $u \in L^{1+q}(\Omega, d\sigma)$ to the integral equation
\begin{equation} \label{int_eq_domain}
u = \g(u^{q}d\sigma) + \g\mu \quad  d\sigma\text{-}a.e.
\end{equation}
where $0<q<1$ and $\sigma, \mu \in \mathcal{M}^{+}(\Omega)$, provided that $G$ is a quasi-symmetric kernel  which satisfies the WMP. 

If $G$ is Green's function associated with $-\Delta$ on $\Omega$, integral equation \eqref{int_eq_domain} is equivalent to the sublinear elliptic boundary value problem
\begin{equation} \label{main_prob_frac_laplacian_domain}
\begin{cases}
-\Delta u = \sigma u^{q} + \mu \quad \text{in} \;\; \Omega, \\
 u = 0  \qquad \qquad \quad \;\; \text{on} \;\; \partial\Omega.
\end{cases}
\end{equation}
As an application, we can deduce necessary and sufficient conditions for existence of a positive finite energy solution $u \in L^{q}_{loc}(\Omega, d\sigma) \cap \dot{W}^{1,2}_{0}(\Omega)$ to equation \eqref{main_eq_frac_laplacian_domain}.

We will need the following result proved in \cite{V17a}, which explicitly characterizes $(p,r)$-weighted norm inequalities
\begin{equation}\label{weighted_norm_ineq}
\big\| \g(f d\sigma) \big\|_{L^{r}(\Omega, d\sigma)} \leq C \| f \|_{L^{p}(\Omega, d\sigma)}, \quad \forall  f \in L^{p}(\Omega, d\sigma),
\end{equation}
where $C$ is a positive constant independent of $f$, in the case $0< r < p$ and $1< p < \infty$, under some mild assumptions on the kernel $G$.

\begin{Thm}[\cite{V17a}]\label{thm_I}
Let $\sigma \in \mathcal{M}^{+}(\Omega)$ and $G$ be a positive quasi-symmetric lower semicontinuous kernel on $\Omega \times \Omega$, which satisfies the WMP.
\begin{itemize}
\item[(i)] If $1 < p < \infty$ and $0 < r < p$, then the $(p,r)$-weighted norm inequality \eqref{weighted_norm_ineq} holds if and only if 
\begin{equation}\label{energy1}
\g\sigma \in L^{\frac{pr}{p-r}}(\Omega, d\sigma).
\end{equation}

\item[(ii)]If $0 < q < 1$ and $q < r < \infty$, then there exists a positive supersolution $u \in L^{r}(\Omega, d\sigma)$ to the homogeneous integral equation \eqref{int_eq_domain} with $\mu=0$, so that  
\begin{equation}\label{int_eq_homo_domain}
u \ge  \g(u^{q}d\sigma) \quad d\sigma\text{-}a.e.
\end{equation}
if and only if weighted norm inequality \eqref{weighted_norm_ineq} holds with $p = \frac{r}{q}$, that is, 
\begin{equation}\label{weighted_norm_ineq_special}
\big\| \g(f d\sigma) \big\|_{L^{r}(\Omega, d\sigma)} \leq C \| f \|_{L^{\frac{r}{q}}(\Omega, d\sigma)}, \quad \forall f \in L^{\frac{r}{q}}(\Omega, d\sigma),
\end{equation}
where $C$ is positive constant independent of $f$, or equivalently
\begin{equation}\label{energy2}
\g \sigma \in L^{\frac{r}{1-q}}(\Omega, d\sigma).
\end{equation}
\end{itemize}
\end{Thm}

The following theorem gives necessary and sufficient conditions for the existence of a positive solution $ u \in L^{1+q}(\Omega, d\sigma)$ to integral equation \eqref{int_eq_domain}. In fact, it is a more general version of Theorem \ref{thm_existence} in the linear case $p=2$.

\begin{Thm}\label{thm_existence_domain}
Let $0< q < 1$ and $\sigma, \mu \in \mathcal{M}^{+}(\Omega)$, and let $G$ be a positive quasi-symmetric lower semicontinuous kernel on $\Omega \times \Omega$, which satisfies the WMP.
Then there exists a positive solution $u \in L^{1+q}(\Omega, d\sigma)$ to integral equation \eqref{int_eq_domain} if and only if 
\eqref{condition1_domain} and \eqref{condition2_domain} hold.
\end{Thm}

\begin{proof}
The sufficiency part is similar to the one of Theorem \ref{thm_existence} when $p=2$, proved by applying Theorem \ref{thm_I} (ii) in the case $r = q+1$ in place of Lemma \ref{bdd_operator}, and replacing Wolff's potentials by potential operators $\g$ associated with the kernel $G$. The necessity part follows immediately from Theorem \ref{thm_I} (ii) in the case $r = q+1$.
\end{proof}

We now apply the above result  to deduce necessary and sufficient conditions for the existence of a positive finite energy solution to equation \eqref{main_eq_frac_laplacian_domain}. As in previous sections, we first make the following observation regarding relation between conditions \eqref{condition1_domain}, \eqref{condition3_domain} and \eqref{condition2_domain}.

\begin{Lem} \label{relation_of_conditions_domain}
Let $0<q<1$ and $\sigma, \mu \in \mathcal{M}^{+}(\Omega)$, and let $G$ be a positive quasi-symmetric lower semicontinuous kernel on $\Omega \times \Omega$, which satisfies the WMP.
Then \eqref{condition1_domain} and 
\eqref{condition3_domain} imply \eqref{condition2_domain}.
\end{Lem}

\begin{proof}
By Theorem \ref{thm_I} (ii) with $r=1+q$, \eqref{condition1_domain} holds if and only if there exists a constant $C$ such that
\begin{equation}\label{weighted_norm_ineq_special2.1}
\big\| \g(f d\sigma) \big\|_{L^{1+q}(\Omega, d\sigma)} \leq C \| f \|_{L^{\frac{1+q}{q}}(\Omega, d\sigma)}, \quad \forall f \in L^{\frac{1+q}{q}}(\Omega, d\sigma).
\end{equation}
Suppose $f$ is any nonnegative bounded measurable function with compact support in $\Omega$. Applying H\"{o}lder's inequality and the weighted norm inequality \eqref{weighted_norm_ineq_special2.1} we have
\begin{equation}\label{est1.1}
\begin{split}
\big\| \g\left( f d\sigma \right) \big\|_{\dot{W}^{1,2}_{0}(\Omega)}^{2}
&= \int_{\Omega} \big| \nabla \g(f d\sigma)\big|^{2}\;dx \\
&= \int_{\Omega} \g(f d\sigma) \cdot f  \;d\sigma \\
&\leq \big\| \g(fd\sigma) \big\|_{L^{1+q}(\Omega,\;d\sigma)}
\big\| f \big\|_{L^{\frac{1+q}{q}}(\Omega,\;d\sigma)}
\\
&\leq  C\big\| f \big\|_{L^{\frac{1+q}{q}}(\Omega,\;d\sigma)}^{2}.
\end{split}
\end{equation}
Since $\mu \in \dot{W}^{-1,2}(\Omega)$, by Tonelli's Theorem and Brezis-Browder theorem (Theorem \ref{representative_thm}), we obtain
\begin{equation}\label{est1.11}
\begin{split}
\bigg| \int_{\Omega} \left( \g\mu \right)f \;d\sigma \bigg| 
&= \bigg| \int_{\Omega} \g \left( f d\sigma \right) \;d\mu \bigg|  \\
&= \Big| \langle \g(fd\sigma), \mu \rangle \Big| \\
&\leq \big\| \g (f d\sigma) \big\|_{\dot{W}^{1,2}_{0}(\Omega)} \big\| \mu \big\|_{\dot{W}^{-1,2}(\Omega)} \\
&\leq c \big\| f \big\|_{L^{\frac{1+q}{q}}(\Omega,\;d\sigma)} \big\| \mu \big\|_{\dot{W}^{-1,2}(\Omega)}.
\end{split}
\end{equation}
Applying a standard density argument, we see that \eqref{est1.11} actually holds for all $f \in L^{\frac{1+q}{q}}(\Omega, d\sigma)$. By duality, taking the supremum over all $f \in L^{\frac{1+q}{q}}(\Omega, d\sigma)$, we get 
\[
\left( \int_{\Omega} \left( \g \mu \right)^{1+q}\;d\sigma \right)^{\frac{1}{1+q}} \leq c \| \mu \|_{\dot{W}^{-1,2}(\Omega)} < +\infty,
\]
which proves the lemma.
\end{proof}

The next lemma shows in particular that conditions \eqref{condition1_domain} and \eqref{condition3_domain} are neccessary for the  existence of a positive finite energy solution to equation \eqref{main_eq_frac_laplacian_domain}.

\begin{Lem} \label{lemma_converse_2_domain}
Let $0<q<1$ and $\sigma, \mu \in \mathcal{M}^{+}(\Omega)$, and let $G$ be Green's function associated with $-\Delta$ on $\Omega$. Suppose there exists a positive supersolution $u \in L^{q}_{loc}(\Omega, d\sigma) \cap \dot{W}^{1,2}_{0}(\Omega)$ to equation \eqref{main_eq_frac_laplacian_domain}. Then
\[
-\Delta u \in \dot{W}^{-1,2}(\Omega) \cap \mathcal{M}^{+}(\Omega),
\]
and hence \eqref{condition3_domain} holds. Moreover, 
$u \in L^{1+q}(\Omega, d\sigma)$ for a quasicontinuous representative of $u$, and consequently \eqref{condition1_domain} holds as well.
\end{Lem}

\begin{proof}
By Schwarz's inequality, for every $\varphi \in C^{\infty}_{0}(\Omega)$ we have 
\[
\big| \langle -\Delta u, \varphi \rangle \big| 
= \bigg| \int_{\Omega} \nabla u \cdot \nabla\varphi \;dx  \bigg|
\leq \| \nabla u \|_{L^{2}(\Omega)} \| \nabla \varphi \|_{L^{2}(\Omega)}.
\]
Hence, $-\Delta u \in \dot{W}^{-1,2}(\Omega)$. Moreover, for every nonnegative $\varphi \in C^{\infty}_{0}(\Omega)$ we have 
\[ 
\langle -\Delta u, \varphi \rangle
=  \int_{\Omega} \nabla u \cdot \nabla\varphi \;dx
\geq \int_{\Omega} u^{q}\varphi \;d\sigma + \int_{\Omega} \varphi \;d\mu 
\geq 0.
\]
This shows that $-\Delta u \in \mathcal{M}^{+}(\Omega)$, from which it follows that \eqref{condition3_domain} holds, and 
\[
d\nu := u^{q}d\sigma \in \dot{W}^{-1,2}(\Omega) \cap \mathcal{M}^{+}(\Omega).
\]
Let $\lbrace \varphi_{j} \rbrace_{1}^{\infty} \subset C_{0}^{\infty}(\Omega)$ be a sequence of nonnegative functions such that $\varphi_{j} \rightarrow u $ in $\dot{W}^{1,2}_{0}(\Omega)$ as $j \rightarrow \infty$. Then
\[
\langle \nu, \varphi_{j} \rangle \leq \int_{\Omega} \nabla u \cdot \nabla \varphi_{j} \;dx \quad \text{for all}\;\; j \in \N.
\]
Hence, 
\[
\langle \nu, u \rangle 
= \lim_{j \rightarrow \infty} \langle \nu, \varphi_{j} \rangle 
\leq \lim_{j \rightarrow \infty} \int_{\Omega} \nabla u \cdot \nabla \varphi_{j} \;dx
=\int_{\Omega} | \nabla u |^{2} \;dx < +\infty.
\]
Applying the Brezis-Browder theorem (Theorem \ref{representative_thm}), for a quasicontinuous representative of $u$, we have 
\[
\int_{\Omega} u^{1+q} \;d\sigma = \int_{\Omega} u \;d\nu = \langle \nu, u \rangle  < +\infty.
\]
Hence, $u \in L^{1+q}(\Omega, d\sigma)$. Consequently, by Theorem \ref{thm_existence_domain}, it follows that \eqref{condition1_domain} holds.
\end{proof}

Finally, if $G$ is Green's function associated with $-\Delta$ on $\Omega$, the next lemma shows in particular that conditions \eqref{condition1_domain} and \eqref{condition3_domain} are sufficient for the existence of a minimal positive finite energy solution to equation \eqref{main_eq_frac_laplacian_domain}.

\begin{Lem} \label{thm_converse_1_domain}
Let $0<q<1$ and $\sigma, \mu \in \mathcal{M}^{+}(\Omega)$, and let $G$ be a positive quasi-symmetric lower semicontinuous kernel on $\Omega \times \Omega$, which satisfies the WMP. Suppose
\eqref{condition1_domain} and \eqref{condition3_domain} hold.
Then there exists a positive solution $w \in L^{q}_{loc}(\Omega, d\sigma)\cap \dot{W}^{1,2}_{0}(\Omega)$ to integral equation \eqref{int_eq_domain}. Moreover, $w$ is a minimal positive solution in the sense that $w \leq u$ q.e. for any positive solution $u \in L^{q}_{loc}(\Omega, d\sigma)\cap \dot{W}^{1,2}_{0}(\Omega)$ to \eqref{int_eq_domain}.
\end{Lem}

\begin{proof}
Since \eqref{condition1_domain} and \eqref{condition3_domain} hold, then by Lemma \ref{relation_of_conditions_domain}, it follows that \eqref{condition2_domain} holds. In view of Theorem \ref{thm_existence_domain}, there exists a positive solution $w \in L^{1+q}(\Omega, d\sigma)$ to \eqref{int_eq_domain}. Obviously, $w \in L^{q}_{loc}(\Omega, d\sigma)$ by H\"{o}lder's inequality. We will show that $w \in \dot{W}^{1,2}_{0}(\Omega)$. By \eqref{est1.1} in the proof of Lemma \ref{relation_of_conditions_domain}, we observe that 
\begin{equation}\label{est1.2}
\big\| \g\left( w^q d\sigma \right) \big\|_{\dot{W}^{1,2}_{0}(\Omega)}
\leq  c \big\| w \big\|_{L^{1+q}(\Omega,\;d\sigma)} 
< +\infty.
\end{equation}
Hence
\[
\begin{split}
\big\| w  \big\|_{\dot{W}^{1,2}_{0}(\Omega)}
&\leq \big\| \g\left( w^q d\sigma \right) \big\|_{\dot{W}^{1,2}_{0}(\Omega)} 
+ \big\| \g\mu \big\|_{\dot{W}^{1,2}_{0}(\Omega)} \\
&= \big\| \g\left( w^q d\sigma \right) \big\|_{\dot{W}^{1,2}_{0}(\Omega)} 
+ \big\| \mu \big\|_{\dot{W}^{-1,2}(\Omega)} \\
&< +\infty,
\end{split}
\]
which proves the lemma. As in the proof of Theorem \ref{thm_converse_1}, minimality of $w$ follows immediately from its construction in Theorem \ref{thm_existence_domain}.
\end{proof}
%%% ----------------------------------------------------------------------------------------------
%%% ----------------------------------------------------------------------------------------------
\section{Uniqueness}
In this section, we establish the uniqueness of positive finite energy solutions to equations \eqref{main_eq_p-lapacain}, \eqref{main_eq_frac_laplacian} and \eqref{main_eq_frac_laplacian_domain}, using the idea used in \cite{CV14a}, namely employing convexity properties of Dirichlet integrals and minimality of such solutions.

\begin{Thm}\label{uniqueness_p-Laplacian}
Let $1<p<n$, $0<q<p-1$ and $\sigma, \mu \in \mathcal{M}^{+}(\mathbb{R}^n)$. Suppose there exists a positive finite energy solution to equation \eqref{main_eq_p-lapacain}. Then such a solution is unique in $\dot{W}^{1,p}_{0}(\R^n)$.
\end{Thm}

\begin{proof}
Suppose $u$ and $v$ are positive finite energy solutions to \eqref{main_eq_p-lapacain}. We start  
with the following  two observations.
We first claim that
\begin{center}
if $u=v$ $d\sigma$-a.e. then $u=v$ as elements of $\dot{W}_{0}^{1,p}(\R^n)$.
\end{center}
To see this, suppose $u=v$ $d\sigma$-a.e., and set 
\[
d\omega := u^{q}d\sigma + d\mu = v^{q}d\sigma + d\mu.
\]
Then, $\omega \in \mathcal{M}^{+}(\R^n)$ and 
\begin{equation}\label{eq1_uniqueness_2}
-\Delta_{p}u = -\Delta_{p}v = \omega \quad \text{in} \;\; \R^n.
\end{equation}
As usual, we may consider quasicontinuous representatives of $u$ and $v$. Then, by Lemma \ref{converse_existence_thm_inhomo},
\[
u,v \in L^{1+q}(\R^n, d\sigma)  \quad \text{and} \quad \omega \in W^{-1,p'}(\R^n) .
\] 
As discussed in Remark \ref{Rem0}, for such a measure $\omega$, a solution $u \in \dot{W}_{0}^{1,p}(\R^n)$ to the equation $-\Delta_{p}u = \omega$\, in $\R^n$, is unique. Hence, $u=v$ q.e., so they coincide as elements of $\dot{W}_{0}^{1,p}(\R^n)$.

Secondly, we claim that
\begin{center}
if $u \geq v$ q.e. then $u=v$ $d\sigma$-a.e. 
\end{center}
Suppose $u \geq v$ q.e. then $u \geq v$ $d\sigma$-a.e. and $u \geq v$ $d\mu$-a.e., since $\sigma$ and $\mu$ are absolutely continuous with respect to $\text{cap}_{p}(\cdot)$. Testing the equations 
\begin{equation}\label{eq_test_1}
\begin{split}
& \int_{\R^n} \vert \nabla u \vert^{p-2} \nabla u \cdot \nabla \phi \; dx \\ = & \int_{\R^n} u^{q}\phi \;d\sigma + \int_{\R^n} \phi \;d\mu, \qquad 
\phi \in \dot{W_{0}}^{1,p}(\R^n),\end{split}
\end{equation}
\begin{equation}\label{eq_test_2}
\begin{split}& \int_{\R^n} \vert \nabla v \vert^{p-2} \nabla v \cdot \nabla \psi \; dx \\ =  &\int_{\R^n} v^{q}\psi \;d\sigma + \int_{\R^n} \psi \;d\mu, \qquad 
\psi \in \dot{W}_{0}^{1,p}(\R^n),
\end{split}
\end{equation}
with $\phi = u$ and $\psi = v$, respectively, where $\omega =u^q \sigma+  \mu \in  \dot{W}^{-1,p'}(\R^n)$, so that Theorem \ref{representative_thm} is applicable for 
quasi-continuous representatives of $u$ and $v$, 
we obtain 

\begin{equation}\label{eq_tested_1}
\int_{\R^n} \vert \nabla u \vert^{p} \; dx 
=  \int_{\R^n} u^{1+q} \;d\sigma + \int_{\R^n} u \;d\mu, \;\; \text{and}
\end{equation}

\begin{equation}\label{eq_tested_2}
\int_{\R^n} \vert \nabla v \vert^{p} \; dx 
=  \int_{\R^n} v^{1+q} \;d\sigma + \int_{\R^n} v \;d\mu.
\end{equation}
Using convexity of the Dirichlet integral $\int_{\R^n} \vert \nabla \cdot \vert^{p}\;dx$
along curves of the type 
\[
\lambda_{t}(x) := \left[ (1-t)v^{p}(x) + tu^{p}(x) \right]^{\frac{1}{p}}, \quad t \in [0,1],
\]
see \cite[Proposition 2.6]{BF14}, we obtain
\[
\begin{split}
\int_{\R^n} \vert \nabla \lambda_{t} \vert^{p} \; dx 
&\leq (1-t) \int_{\R^n} \vert \nabla v \vert^{p} \;dx + t \int_{\R^n} \vert \nabla u \vert^{p} \; dx \\
&= t \left( \int_{\R^n} \vert \nabla u \vert^{p} \; dx - \int_{\R^n} \vert \nabla v \vert^{p} \; dx\right) + \int_{\R^n} \vert \nabla v \vert^{p} \; dx.
\end{split}
\]
Notice that $\lambda_{0} = v$. By \eqref{eq_tested_1} and 
\eqref{eq_tested_2}, we get 
\[
\int_{\R^n} \frac{\vert \nabla \lambda_{t} \vert^{p} - \vert \nabla \lambda_{0} \vert^{p}}{t} \; dx
\leq \int_{\R^n} (u^{1+q} - v^{1+q}) \;d\sigma + \int_{\R^n} (u-v) \;d\mu.
\]
Using the inequality
\[
\vert a \vert^{p} - \vert b \vert^{p} \geq p \vert b \vert^{p-2} b \cdot (a-b) \qquad \text{for} \;\; a,b \in \R^n,
\]
we deduce that
\[
\vert \nabla \lambda_{t} \vert^{p} - \vert \nabla \lambda_{0} \vert^{p} 
\geq p \vert \nabla \lambda_{0} \vert^{p-2} \nabla \lambda_{0} \cdot (\nabla \lambda_{t}-\nabla \lambda_{0}),
\]
and hence 
\begin{equation}\label{est1_uniqueness_thm2}
p \int_{\R^n} \vert \nabla v \vert^{p-2} \nabla v \cdot \frac{\nabla(\lambda_{t} - \lambda_{0})}{t}  \; dx
\leq \int_{\R^n} (u^{1+q} - v^{1+q}) \;d\sigma + \int_{\R^n} (u-v) \;d\mu.
\end{equation}
Testing \eqref{eq_test_2} with $\psi = \lambda_t - \lambda_0 \in \dot{W}^{1,p}_{0}(\R^n)$, we obtain 
\begin{equation}\label{eq_test_2.1}
\int_{\R^n} \vert \nabla v \vert^{p-2} \nabla v \cdot \nabla (\lambda_t - \lambda_0) \; dx 
=  \int_{\R^n} v^{q} (\lambda_t - \lambda_0) \;d\sigma + \int_{\R^n}  (\lambda_t - \lambda_0) \;d\mu.
\end{equation}
Thus, by \eqref{est1_uniqueness_thm2} and \eqref{eq_test_2.1}, we have
\begin{equation}\label{est2_uniqueness_thm2}
p \int_{\R^n} v^{q}  \frac{\lambda_t - \lambda_0 }{t}\;d\sigma
+ p \int_{\R^n} \frac{\lambda_t - \lambda_0 }{t}\;d\mu
\leq \int_{\R^n} (u^{1+q} - v^{1+q}) \;d\sigma 
+ \int_{\R^n} (u-v) \;d\mu.
\end{equation}
Since $u \geq v$ q.e. then $\lambda_t \geq \lambda_0 \;\; d\sigma$-a.e. and $\lambda_t \geq \lambda_0 \;\; d\mu$-a.e. Applying Fatou's Lemma, we obtain
\begin{equation}\label{est3_uniqueness_thm2}
\int_{\R^n} v^{q}  \frac{u^p - v^p}{v^{p-1}} \;d\sigma 
\leq \liminf_{t \rightarrow 0} p \int_{\R^n} v^{q}  \frac{\lambda_t - \lambda_0}{t} \;d\sigma
\end{equation}
and 
\begin{equation}\label{est4_uniqueness_thm2}
\int_{\R^n}  \frac{u^p - v^p}{v^{p-1}} \;d\mu
\leq \liminf_{t \rightarrow 0} p \int_{\R^n} \frac{\lambda_t - \lambda_0}{t} \;d\mu.
\end{equation}
Since \eqref{est2_uniqueness_thm2} holds for all $t \in [0,1]$ then by 
\eqref{est3_uniqueness_thm2} and \eqref{est4_uniqueness_thm2}, we arrive at 
\[
\int_{\R^n}  \frac{u^p v^q}{v^{p-1}} - v^{1+q} \;d\sigma  +
\int_{\R^n}  \frac{u^p}{v^{p-1}} - v \;d\mu
\leq \int_{\R^n} (u^{1+q} - v^{1+q}) \;d\sigma + \int_{\R^n} (u-v) \;d\mu,
\] 
that is,
\[
\int_{\R^n}  \frac{u^p v^q}{v^{p-1}} - u^{1+q} \;d\sigma  +
\int_{\R^n}  \left( \frac{u^p}{v^{p-1}}  - u \right) \;d\mu
\leq  0.
\] 
Here both integrals on the left-hand side are nonnegative since $u \geq v$ $d\sigma$-a.e. and $u \geq v$ $d\mu$-a.e. Indeed,
\[
\begin{split}
\int_{\R^n}  \frac{u^p v^q}{v^{p-1}} - u^{1+q} \;d\sigma 
&= \int_{\R^n}  \frac{u^p v^q - u^{1+q} v^{p-1}}{v^{p-1}}  \;d\sigma  \\
& = \int_{\R^n}  \frac{u^{1+q}v^{q} (u^{p-1-q} - v^{p-1-q})}{v^{p-1}}\;d\sigma  \\
&\geq 0
\end{split}
\]
and 
\[
\int_{\R^n}  \left( \frac{u^p}{v^{p-1}}  - u \right) \;d\mu 
= \int_{\R^n}  \frac{u^{p} - uv^{p-1}}{v^{p-1}}\;d\mu \geq 0.
\]
Therefore, both integrals must vanish, and thus $u=v$ $\;d\sigma$-a.e. and 
$u=v$ $\;d\mu$-a.e. In particular, this proves the second claim.

Now, suppose $\tilde{w}$ is any positive finite energy solution to \eqref{main_eq_p-lapacain}. Then 
\[
  \tilde{w} \geq w \quad \text{q.e.},
\]
where $w$ is the minimal positive finite energy solution to \eqref{main_eq_p-lapacain} constructed in Theorem \ref{thm_main}. 
Applying the second claim above, we have 
\[
\tilde{w} = w  \qquad d\sigma\text{-a.e.},
\]
and hence, by the first claim, they coincide as elements of $\dot{W}^{1,p}_{0}(\R^n)$.
\end{proof}

By a slight modification of the argument above, we can establish the uniqueness of a positive finite energy solution to equation \eqref{main_eq_frac_laplacian} when $0<\alpha  \leq 1$.

\begin{Thm}\label{uniqueness_fractional}
Let $0<q<1$, $0< \alpha \leq 1$, and $\sigma, \mu \in \mathcal{M}^{+}(\mathbb{R}^n)$. Suppose there exists a positive finite energy solution to equation \eqref{main_eq_frac_laplacian}. Then such a solution is unique in $\dot{H}^{\alpha}(\R^n)$.
\end{Thm}
\begin{proof}
When $\alpha =1$, this follows from Theorem \ref{uniqueness_p-Laplacian} in the case $p=2$. If $0<\alpha<1$, we use the same argument as in the proof of Theorem \ref{uniqueness_p-Laplacian} together with convexity of Gagliardo seminorms established in \cite{BF14}, instead of convexity of the Dirichlet integrals $\int_{\R^n} |\nabla \cdot|^{p}\;dx$.
\end{proof}

Since convexity of the Dirichlet integrals $\int_{\Omega} |\nabla \cdot|^2\;dx$ is also available on arbitrary nonempty open sets $\Omega \subset \R^n$ (see \cite{BF14}), we may argue in the same way as in the proof of Theorem \ref{uniqueness_p-Laplacian} in the case $p=2$ to obtain the following theorem on the uniqueness of a positive finite energy solution to equation \eqref{main_eq_frac_laplacian_domain}.

\begin{Thm}\label{uniqueness_domain}
Let $0<q<1$, and let $\sigma, \mu \in \mathcal{M}^{+}(\Omega)$. Suppose there exists a positive finite energy solution to equation \eqref{main_eq_frac_laplacian_domain}. Then such a solution is unique
in $\dot{W}^{1,2}_{0}(\Omega)$.
\end{Thm}
%%% ----------------------------------------------------------------------------------------------
%%% ----------------------------------------------------------------------------------------------

%%% ----------------------------------------------------------------------------------------------
%%% ----------------------------------------------------------------------------------------------
\end{document}